\documentclass[12pt]{amsart}

\usepackage{graphicx}

\usepackage{subcaption} %para facer a artistada coas figuras nun grid
\usepackage{quiver}
\usepackage{adjustbox}
\usepackage{tikz-cd}
\usepackage{hyperref}
\hypersetup{bookmarks=true,
	unicode=true,
	colorlinks=true,
	citecolor=black,
	linkcolor=black,
	urlcolor=black,
	plainpages=false,
	pdfpagelabels=true}

 \usepackage{url}	%url adresses
 \allowdisplaybreaks %align may cross pages

\usepackage{pgf}

\usepackage{xcolor}

\newcommand{\DM}[1]{{\small\color{violet}#1}}

\usepackage{comment} % para comentar trozos grandes de codigo

\newtheorem{teo}{Theorem}[section]

\newtheorem{corollary}[teo]{Corollary}

\newtheorem{lemma}[teo]{Lemma}
\newtheorem{prop}[teo]{Proposition}

\theoremstyle{definition}
\newtheorem{definition}[teo]{Definition}
\newtheorem{defi}[teo]{Definition}
\newtheorem{example}[teo]{Example}

\newtheorem{ex}[teo]{Example}

\theoremstyle{remark}
\newtheorem{remark}{Remark}
\newtheorem{rem}{Remark}

\numberwithin{figure}{section}%figure numbering

\newcommand{\Ba}{\mathcal{B}}

\newcommand{\CC}{\mathcal{C}}

\newcommand{\ccat}{\mathop{\mathrm{ccat}}}
\newcommand{\cD}{\mathrm{cD}}

\newcommand{\co}{\colon}

\newcommand{\cTC}{\mathop{\mathrm{cTC}}}

\newcommand{\D}{\mathcal{D}}

\newcommand{\Dc}{\mathcal{D}}

\newcommand{\E}{\mathcal{E}}

\newcommand{\id}{\mathrm{id}}

\newcommand{\sC}{\mathrm{secat}}

\newcommand{\sG}{\mathrm{SG}}

\newcommand{\Ab}{\textbf{\textsf{Ab}}}

\newcommand{\Uc}{\mathcal{U}}
\newcommand{\wH}{\wtilde{H}}

\newcommand{\wtilde}{\widetilde}
\newcommand{\Ef}{\mathrm{E}_F}

\newcommand{\cpl}{\mathrm{c.p.l}}
\newcommand{\Pa}{\mathrm{P}}

\begin{document}

\title[Weak and Strong Fibrations of Functors]{Weak and Strong Fibrations of Functors and Motion Planning}

\author[I. Carcac\'ia]{%
	Isaac Carcac\'ia-Campos %etc.
}

 \address{%
           	Isaac Carcac\'ia Campos
            \\
              Departamento de Matem\'aticas, Universidade de Santiago de Compostela, Universidade de Vigo, 15782-SPAIN}
               \email{isaac.c.campos@usc.es}

\author[E. Mac\'ias]{%
	Enrique Mac\'ias-Virg\'os %etc.
}

 \address{%
           	Enrique Mac\'ias-Virg\'os \\
              CITMAga, Departamento de Matem\'aticas, Universidade de Santiago de Compostela, 15782-SPAIN}
               \email{quique.macias@usc.es}

\author[D. Mosquera]{%
	David Mosquera-Lois %etc.
}

 \address{%
           	David Mosquera-Lois \\
              Departamento de Matemáticas, Universidade de Vigo, 15782-SPAIN}
               \email{david.mosquera.lois@uvigo.gal}

\begin{abstract} We develop a homotopical framework for small categories that extends classical invariants of algebraic topology to the categorical setting. 
Our approach is based on the construction of a genuine path category, obtained through a zig--zag and localization procedure, which allows us to define strong and weak fibrations for functors. 
We establish their basic properties, introduce a fibrant replacement for functors, and extend homotopical invariants such as the \v{S}varc genus and sectional category to small categories. 
Finally, we apply this framework to motion planning in small categories, providing categorical analogues of Farber’s topological complexity while removing finiteness constraints typical of existing approaches.
\end{abstract}




\maketitle

%\tableofcontents

\section*{Introduction}

The study of homotopical invariants has played a central role in algebraic topology since its origins. 
Classical notions such as the Lusternik--Schnirelmann category (LS-category)~\cite{LS-CATEGORY}, topological complexity~\cite{FARBER}, the \v{S}varc genus and sectional category~\cite{Svarc}, and the more recent notion of homotopic distance~\cite{MAC-MOSQ}, provide quantitative tools for measuring the complexity of spaces and maps. 
These invariants are particularly relevant in applications such as motion planning in robotics, where their values for a configuration space determine the intrinsic difficulty of designing a continuous motion planning algorithm~\cite{MEDIT}.

In recent years, several authors have explored categorical analogues of these invariants. 
Tanaka studied LS-category for small categories~\cite{TANAKA}. 
Subsequently, the two last authors of present paper introduced the notions of topological complexity and homotopic distance for functors~\cite{MAC-MOSQ-FUNCT}. 
More recently, these ideas have been extended and refined in various directions~\cite{CMM,Cohomologysvarcgenus}.

All these works build on Lee’s approach to homotopies in categories~\cite{LEE} and on Grothendieck’s theory of fibrations~\cite{Vistoli}. 
Nevertheless, these approaches do not fully capture the homotopical behaviour required in the categorical setting, which limits the study of motion planning invariants for small categories.

The aim of this paper is to overcome these limitations by developing a systematic homotopical framework for small categories, extending classical invariants to this combinatorial setting. 
Our starting point is Hoff’s approach to homotopy theory in categories~\cite{HOFF}, based on the construction of a \emph{path category}. 
However, Hoff’s construction proves too rigid for our purposes. 
This motivates our definition of a novel genuine \emph{path category} associated with a small category, obtained via a zig--zag construction followed by a localization procedure. 
This provides a flexible and well-behaved model for categorical paths, which in turn allows us to introduce appropriate notions of fibration.

The main contributions of this work may be summarized as follows:
\begin{itemize}
    \item We introduce a path category $\Pa \CC$ for any small category $\CC$, together with an endpoint functor that plays the role of the classical path fibration (Theorem~\ref{Initial_Final_Path_fibration}).
    \item We define notions of \emph{strong} and \emph{weak} fibrations for functors, inspired by the classical lifting properties in topology, and establish their first properties, including relations between them (Theorem~\ref{under_length_preserving_strong_weak_equivalence}), and their behaviour under pullbacks and composition.
    \item We construct the \emph{fibrant replacement} of a functor (Theorem~\ref{functor_factorisation_through_fibration}), showing that any functor factors as a weak homotopy equivalence followed by a fibration.
    \item Using these tools, we extend classical invariants such as the \v{S}varc genus, sectional category, LS-category, categorical complexity, and homotopic distance to the setting of small categories, introducing both strong and weak versions.
    \item Finally, we apply this framework to \emph{motion planning in small categories}, obtaining categorical analogues of Farber’s topological complexity and proving relations among the different invariants (Theorem~\ref{teo:Svarc_Genus_Homotopic_Distance}).
\end{itemize}

Furthermore, by working in the setting of small categories we are able to overcome the finiteness condition that is required for fibrations in the simplicial context~\cite{Simplicial_fibrations}. 
This yields a more flexible and general theory, extending the scope of homotopical invariants beyond the restrictions of the simplicial framework.

\subsection*{Organization of the paper}  
Section~\ref{sec:path-category} introduces the construction of the path category and establishes its basic properties. 
Section~\ref{sec:fibrations} defines strong and weak fibrations and proves their main structural results. 
Section~\ref{sec:fibrant_replacement_svarc} presents the fibrant replacement of a functor and extends the notion of \v{S}varc genus to categories. 
Finally, Section~\ref{sec:motion_planning} applies our framework to motion planning in small categories, introducing weak and strong versions of LS-category, categorical complexity, and homotopic distance, and establishing relations among these invariants.

\subsection*{Notation and conventions} 
All categories considered are assumed to be small. 
Given two small categories $\CC$ and $\Dc$ we denote their product by $\CC \times \Dc$. 
Its objects are pairs consisting of an object of $\CC$ and an object of $\Dc$, and its morphisms are pairs of morphisms in $\CC$ and $\Dc$, respectively.

\section{The right notion for a path category}\label{sec:path-category}

\subsection{Definition of the path category} 
For each natural number $m$ we define the {\em zig-zag category of length $m$}, denoted by $\mathbb{I}_m$,  as the category whose objects are the positive integers from $0$ to $m$ and morphisms $i \rightarrow i+1$ if $i$ is even and $i \leftarrow i+1$ if $i$ is odd.

Let $\CC$ be a small category. A \emph{path of length $m$} in $\CC$ is a functor $I: \mathbb{I}_m \rightarrow \CC$. For a path of length $m$ $I: \mathbb{I}_m \rightarrow \CC$ we refer to the object $I(0)$ as the initial point or object of the path $I$ and to $I(m)$ as the end point or object of the path. As an abuse of notation, we use the term {\em path} to  not specify the length.  A category $\CC$ is \emph{connected} if for every pair of objects $\CC$ and $\Dc$ there is path such that the initial object is $\CC$ and the end point is $\Dc$.

We define $\mathcal{P}{\CC}$ as the category whose objects are paths and a morphism between two paths $I \colon \mathbb{I}_{m} \rightarrow \CC$ and $J \colon \mathbb{I}_{n} \rightarrow \CC$ is a pair $(F,\alpha)\colon I \rightarrow J$ where \begin{itemize}
    \item $F \colon \mathbb{I}_m \rightarrow \mathbb{I}_n$ is a functor such that $F(0)=0$ and $F(m)=n$
    \item and $\alpha$ is a natural transformation $\alpha \colon I \implies J \circ F$.
\end{itemize}   
%In a more diagrammatic fashion we can see a morphism between path as:
% https://q.uiver.app/#q=WzAsMyxbMCwwLCJcXG1hdGhiYntJfV97bl8xfSJdLFsyLDAsIlxcbWF0aGJie0l9X3tuXzF9Il0sWzEsMSwiXFxDQyJdLFswLDEsIkYiXSxbMCwyLCJJXzEiLDJdLFsxLDIsIklfMiJdLFs0LDUsIlxcYWxwaGEiLDEseyJzaG9ydGVuIjp7InNvdXJjZSI6MjAsInRhcmdldCI6MjB9fV1d
\begin{comment}
 
\[\begin{tikzcd}[ampersand replacement=\&]
	{\mathbb{I}_{m}} \&\& {\mathbb{I}_{n}} \\
	\& \CC
	\arrow["F", from=1-1, to=1-3]
	\arrow[""{name=0, anchor=center, inner sep=0}, "{I}"', from=1-1, to=2-2]
	\arrow[""{name=1, anchor=center, inner sep=0}, "{J}", from=1-3, to=2-2]
	\arrow["\alpha"{description}, Rightarrow, from=0, to=1]
\end{tikzcd}\]
   
\end{comment}

\begin{comment}
 
\begin{rem}
    A functor $F \colon \mathbb{I}_m \rightarrow \mathbb{I}_n$ a functor such that $F(0)=0$ and $F(m)=n$ is always surjective and therefore $m\geq n$. Moreover it has to respect the order of the natural numbers in the sense that if $i \leq j\leq m$ then $F(i) \leq F(j)$.
\end{rem}
   
\end{comment}

The category $\mathcal{P}{\CC}$ has (finite) paths as objects. However, it has some issues that we need to fix. For example, two constant paths with the same constant value but different length are different, but from a homotopic point of view, they must be the same since one may be seen as a reparameterization of the other. We solve this difficulty by localizing.

We define the set $W \subset \mathrm{Mor}(\mathcal{P}\CC)$ as the set of all morphisms between paths such that the natural transformation is the identity, i.e. where $I=J \circ F$ and therefore the morphism has the form $(F,1_{I})$. With this we can define our path category using localization \cite{nlab:category_of_fractions}.

\begin{comment}
   
\begin{defi}
    Let $\CC$ be a category and $W$ a class os morphism. The localization category of $\CC$ by $W$, if it exists, is a category $\CC[W^{-1}]$ with the following universal property: there is a universal functor $\pi \colon \CC \rightarrow \CC[W^{-1}]$ such that for each $w \in W$ $\pi(w)$ is an isomorphism and for any other functor $F \colon C \rightarrow D$ such that $F(w)$ is an isomorphism for each $w \in W$ there is another funtor $\bar{F} \colon \CC([W^{-1}] \rightarrow D$ such that $F=\bar{F} \circ \pi$.
\end{defi}

%\begin{comment}
\begin{rem}
    It is well known that the category $\CC[W^{-1}]$ is unique up to equivalence and has a model given by the category whose objects are the same as $\CC$ and the morphism between two objects are a zig zag or morphism where the right one are morphism in $\CC$ and the left one are morphism in $W$.
\end{rem}
%\end{comment}
 
\end{comment}

\begin{defi}
Let $\CC$ be a small category.  We define the {\em path category} of $\CC$, denoted by $\Pa \CC$, as the localization $\mathcal{P} \CC[W^{-1}]$.    
\end{defi}

We highlight the differences between our approach and the construction of the path category due to Hoff \cite{HOFF} in one example. Recall that the definition of path category of a small category $\CC$ by Hoff is the category of finite functors of the form $\CC^\mathbb{I_\infty}$ where $\mathbb{I_\infty}$ is the infinity zig-zag category generated by the following diagram:
$$
% https://tikzcd.yichuanshen.de/#N4Igdg9gJgpgziAXAbVABwnAlgFyxMJZABgBpiBdUkANwEMAbAVxiRGJAF9T1Nd9CKAIzkqtRizZCuPEBmx4CRAEyjq9Zq0QhlM3goFEAzGvGa2AHQtQIOBJzEwoAc3hFQAMwBOEALZIyEBwIJCFuTx9-RFUgkMQw2W8-JBjgpCMHTiA
\begin{tikzcd}
0 \arrow[r] & 1 & 2 \arrow[l] \arrow[r] & \dots
\end{tikzcd}
$$
and $F \colon \mathbb{I}_\infty \rightarrow \CC$  is finite if it stabilizes; that is, if there exists a non-negative integer $n$ such that for all $m\geq n$ we have that $F(m)=F(n)$ and every morphism is mapped to the identity of $F(n)$.

\begin{example}\label{ex:path_category_different}
    Let $\mathcal{S}^1$ be the category generated by the following diagram:
    $$
    % https://tikzcd.yichuanshen.de/#N4Igdg9gJgpgziAXAbVABwnAlgFyxMJZABgBpiBdUkANwEMAbAVxiRAA8QBfU9TXfIRQBGclVqMWbAJ7dxMKAHN4RUADMAThAC2SMiBwQkokACMYYKEgDM++s1aIQa7r2dbdifYePVzlpABaW2p7KSdFEGoGOnMGAAV+PAI2DSxFAAscOS4gA
\begin{tikzcd}
x \arrow[r, "f", bend left] \arrow[r, "g"', bend right] & y
\end{tikzcd}
    $$
    Consider the paths $I\colon \mathbb{I}_5 \rightarrow \mathcal{S}^1$ and $I'\colon \mathbb{I}_3 \rightarrow \mathcal{S}^1$:
    $$
    % https://tikzcd.yichuanshen.de/#N4Igdg9gJgpgziAXAbVABwnAlgFyxMJZABgBoBGAXVJADcBDAGwFcYkQAPEAX1PU1z5CKchWp0mrdgE8efEBmx4CRAExiaDFm0Sc5-JUKJli4rVN1deBwSpGlTmyTpCzrCgcuHJ1jidvYreUVbbwBmBzNnGX0PQztkABZIpwDLHnEYKABzeCJQADMAJwgAWyQyEBwIJFEQACMYMCgK1IsQAtjisqR1KprEOvMXbJAaRnpGxgAFTyNdIqxsgAscLpLyxAj+pGT-ds73bs2AVhpq3baXTvHJmBm5uxBFlbWjjaQzncQANiv2Q7yY5IADs5wGfxAEyms3iwmeS1WY32Iwy3CAA
\begin{tikzcd}
x \arrow[r, "f"] & y & x \arrow[l, "f"'] \arrow[r, "f"] & y & x \arrow[l, "g"'] \\
x \arrow[r, "f"] & y & x \arrow[l, "g"']                &   &                  
\end{tikzcd}
$$
There is a morphism in $\Pa \mathcal{S}^1$ defined by:
$$
% https://tikzcd.yichuanshen.de/#N4Igdg9gJgpgziAXAbVABwnAlgFyxMJZABgBoBGAXVJADcBDAGwFcYkQAPEAX1PU1z5CKchWp0mrdgE8efEBmx4CRAExiaDFm0QhZvfkqFEyxcVqm6uBhQOXDkos5sk69cw4JUp1zidvZreUUvBwBmUj8LN31gu2MUABZI81dAj1sjb2QIqhcA3VjPeyJkvP9LTh5xGCgAc3giUAAzACcIAFskMhAcCCRREAAjGDAobvzK5oy2zqR1Xv7EQcZ6EcYABXjvEFasOoALHBBJt3IAfSKQWa7ECMWkZIq3aZsbpABWGj7H0-ZpmirdZbLLCXb7I4zdq3L4PRAANj+ule8neiAA7N8loiQECYJttmC9odjkiQHUoXNEAsfogABxki5XNEATixSAZz3YFMBa3xINC7GJkLe0KQ91pPWi7AuQRaYsQT1pg2luiZlJh7OpZJR8qpONpnNVIHVoqpmLhLO4lG4QA
\begin{tikzcd}
x \arrow[r, "f"] \arrow[d, "1_x"] & y \arrow[d, "1_y"] & x \arrow[l, "f"'] \arrow[r, "f"] \arrow[d, "f"] & y \arrow[d, "1_y"] & x \arrow[l, "g"'] \arrow[d, "1_x"] \\
x \arrow[r, "f"]                  & y                  & y \arrow[l, "1_y"'] \arrow[r, "1_y"]            & y                  & x \arrow[l, "g"']          
\end{tikzcd}
$$
But there is no such morphism in the category $\CC^\mathbb{I_\infty}$ defined by Hoff. 
\end{example}

\subsection{Functoriality of the path construction}
In this subsection we establish that the path construction behaves functorially. This will allow us to treat the passage from a category to its path category as a genuine functor between categories, and to understand how path categories interact with categorical products.

\begin{lemma}
    If $\CC$ is a small category, then its path category $\Pa \CC$ is also small.
\end{lemma}

\begin{proof}
We begin by proving that $\mathrm{P} \CC$ is a small category. Recall that the objects of $\mathrm{P} \CC$ are  functors $\mathbb{I}_m \to C$ for some $m \in \mathbb{N}$. 

Let us begin by fixing $m$.  Giving a functor $\mathbb{I}_m \rightarrow \CC$ amounts to choosing a finite tuple of objects and a finite tuple of morphisms from the sets $\mathrm{Ob}(C)$ and $\mathrm{Mor}(C)$ (respectively). Therefore, for each $m$, the collection 
$\mathrm{Fun}(\mathbb{I}_m,\CC)$ of functors is a set. 
Finally, the class of all paths is obtained as a countable union
\[
\mathrm{Ob}(\Pa \CC) \;=\; \bigcup_{m \in \mathbb{N}} \mathrm{Fun}(\mathbb{I}_m,\CC),
\]
which is still a set. Hence the objects of $\Pa \CC$ form a set.

Morphisms in $\mathcal{P} \CC$ are pairs $(F,\alpha)$, where $F : I_m \to I_n$ is a 
functor between finite zig–zag categories and $\alpha : I \Rightarrow J \circ F$ 
is a natural transformation. These data are again specified by finitely many 
objects and morphisms of $C$, so $\mathrm{Mor}(P C)$ is a set. Thus $\mathcal{P} \CC$ is small. 

Finally, the localization $\Pa \CC = \mathcal{P} \CC[W^{-1}]$ has the same set of 
objects as $\Pa \CC$, and morphisms are represented by finite zig–zags of arrows in 
$P C$ with arrows in $W$ reversed. Since $\Pa \CC$ is small, the collection of such 
zig–zags is also a set. Therefore $\Pa \CC$ is small.
\end{proof}

\begin{prop}
    There is a functor $\mathrm{P} \colon \textbf{cat} \rightarrow \textbf{cat}$ that takes a category $\CC$ to the category $\mathrm{P} \CC$.
    \begin{proof}
        Suppose that we have a functor $G \colon \CC \rightarrow \Dc$ we have to define a functor $\mathrm{P}F \colon \mathrm{P} \CC \rightarrow \mathrm{P} \Dc$. Take a path $I \colon \mathbb{I}_n \rightarrow \CC$ we have an object $\mathrm{P}G (I)= G \circ I$. Now suppose that you have a morphism of paths in $\CC$ given by $(F,\alpha) \colon I_1 \rightarrow I_2$ we have a morphism of paths in $\Dc$ given by $(F,1_G\alpha)$ where $1_G\alpha$ denotes the horizontal composition of natural transformations. This defines a functor between the localizing categories since every morphism in $W_\CC$ is of the form $(F,1_{I})$ and it goes to an element $(F,1_G1_I)=(F,1_{G\circ I}) \in W_\Dc$. 
    \end{proof}
\end{prop}

To simplify the treatment of the product of categories later, we state the following lemma, which follows from the properties of localization of categories.
\begin{lemma}\label{Product}
    Let $\Ba_1$ and $\Ba_2$ be two categories. There is an equivalence between $\Pa (\Ba_1 \times \Ba_2)$ and $ \Pa(\Ba_1) \times \Pa(\Ba_2)$.
\end{lemma}

\subsection{The path fibration and auxiliary lemmas}\label{subsec:motion_planning_functors}

From the construction of the path category we can also build two functors that relate the path category with the original one.

%\begin{defi}
    Let $\CC$ be a small category.  We define the functors $\pi_0^* \colon \Pa \CC \rightarrow \CC$ and $\pi_1^* \colon \Pa \CC \rightarrow \CC$ as follows. Let $I \colon \mathbb{I}_m \rightarrow \CC$ be a path of length $m$. Then $\pi^*_0(I)=I(0)$, $\pi^*_0(F,\alpha)=\alpha(0)$ and $\pi^*_1(I)=I(m)$, $\pi^*_1(F,\alpha)=\alpha(m)$. We define the {\em initial and final object functors} $\pi_0 \colon \Pa  \CC \rightarrow \CC$ and $\pi_1 \colon \Pa  \CC \rightarrow \CC$ by the universal property of the localization $\Pa \CC=\mathcal{P} \CC[W^{-1}]$ \cite{nlab:category_of_fractions}.  
We define the {\em endpoint functor} or {\em path fibration} $\pi \colon \Pa\CC \rightarrow \CC \times \CC$ as the product $(\pi_0,\pi_1)$.

%\end{defi}

\begin{lemma}\label{Conmutativo_caminos}
    Let $\CC$ and $\Dc$ be small categories and let $G \colon \CC \rightarrow \Dc$ be a functor. Then $\pi_j \circ \Pa G= G \circ \pi_j$ for $j=0,1$.
\end{lemma}
\begin{proof}
    we will show it for the case $j=0$ since the other case is analogous. Take a path $I \colon \mathbb{I}_m \rightarrow \CC$. We have that $\pi_0 (\Pa G)(I)=\pi_0 (G \circ I)=G (I(0))= G (\pi_0(I))=G \circ \pi_0(I)$.
    %In that case it follows from the fact that $\pi_0 \circ \Pa F (I)$ is the initial object in the path $\Pa F (I)$, i.e. is the object $F (I_0)$ where $I_0$ is the initial object of the path $I$ but this is the same as $ F \circ \pi_0 (0)$.
\end{proof}

The following lemma allows us to simplify the number of cases in future proofs and to define without losing generality some constructions applied to paths.

\begin{lemma}
    For each path $I \colon \mathbb{I}_m \rightarrow \CC$ with odd length $n$ there is a equivalent path $I' \colon \mathbb{I}_{m+1} \rightarrow \CC$ with even length.
\end{lemma}

\begin{proof}
    We can take the path $I'$ as:
    $$% https://tikzcd.yichuanshen.de/#N4Igdg9gJgpgziAXAbVABwnAlgFyxMJZABgBpiBdUkANwEMAbAVxiRAEkAKYgShAF9S6TLnyEUARnJVajFmy4S+g4djwEiAJmnV6zVohAAdI1Ag4EKkBjViiAZh2z9CzgFtlQ6yPXjkAFic9eUMuDwEZGCgAc3giUAAzACcINyQyEBwIJAkrZNSkbUzsxFyvfLTEIqykezyUysDi2t05AxAJAH1gMJ5+CP4gA
\begin{tikzcd}
I(0) \arrow[r] & I(1) & \dots \arrow[l] \arrow[r] & I(m) & I(m) \arrow[l, "1_{I(m)}"]
\end{tikzcd}$$
In that case we have a functor $F \colon \mathbb{I}_{m+1} \rightarrow \mathbb{I}_m$ with $F(i)=i$ if $i \leq m+1$ and $F(m+1)=m$ with $F(i_m)=1_{I(m)}$.
\end{proof}

In the future we will always assume for simplicity that all path have an even length. With this in mind we can define the following two operations.

%\begin{defi}
    Let $I \colon \mathbb{I}_m \rightarrow \CC$ and $J \colon \mathbb{I}_n \rightarrow \CC$ be two paths with $m$ and $n$ even and such that $I(m)=J(0)$. We define the {\em inverse path} to $I$, $I^{-1} \colon \mathbb{I}_m \rightarrow \CC$, as $I(i)=I(m-i)$. We define the {\em concatenation} of $I$ and $J$, $J \circ I \colon \mathbb{I}_{m+n} \rightarrow \CC$, as $J \circ I(i)= I(i)$ if $i \leq m$ and $J \circ I(i)= J(i-m)$ if $i \geq m$.
%\end{defi}

\section{Fibrations}\label{sec:fibrations}

\subsection{Homotopy between functors}
We introduce two notions of homotopy between functors. We begin by recalling the notion of homotopy (we refer to it as {\em strong homotopy}) in small categories introduced by Lee (\cite{LEE}). Then we introduce a novel and weaker notion of homotopy (we refer to it as {\em weak homotopy}) inspired on the one introduced by Hoff \cite{HOFF} but with the difference that we work with a different category of paths and explained in Example \ref{ex:path_category_different}.

\begin{defi}
    Let $F,G\colon \CC\to \Dc$ be two functors between small categories. We say that $F$  and $G$ are {\em strong homotopic}, denoted by $F \simeq G$,  if, for some $m\geq 0$, there exists a functor $H\colon \CC \times \mathbb{I}_m \rightarrow \D$, called a homotopy (of length $m$), such that $H(-,0)=F$ and $H(-,m)=G$. We will denote the fact that two functors $F$ and $G$ are strong homotopic with $F \simeq G$. %Alternatively, the functors $F,G\colon \CC\to \Dc$  are strong homotopic  if there is a finite sequence of functors $F_0,\dots,F_m\colon \CC\to \D$, with $F_0=F$  and $F_m=G$, such that for each $i\in \{0,\ldots,m-1\}$ there is a natural transformation either between $F_i$ and $F_{i+1}$ or between $F_{i+1}$ and $F_i$ (see \cite{CMM}).
\end{defi}

Strong homotopy between functors proves to be too rigid for developing motion planning in small categories since with this we do not have the usual equivalence between homotopic sections to $\Delta \colon \CC \rightarrow \CC \times \CC$ and sections to the path fibration $\pi \colon \Pa \CC \rightarrow \CC$. That is why  we introduce weak homotopies in the spirit of \cite{HOFF} and \cite{Simplicial_fibrations}.

\begin{defi}
 Let $F,G\colon \CC\to \Dc$ be two functors between small categories. The functors $F$ and $G$ are weak homotopic, denoted $F \simeq_w G$, (we also say there is a weak homotopy between them) if there exists a functor $H \colon \CC \rightarrow \mathrm{P} \Dc $ such that $\pi_0 \circ H=F$ and $\pi_1 \circ H=G$.  
\end{defi}

Observe that being strong or weak homotopic defines an equivalence relation among functors. 

\begin{comment}

\begin{prop}
    Being strong or weak homotopic is an equivalence relation 
\end{prop}

\begin{proof}
    The part of being strong follows from the fact that strong homotopies can be concatenated and inverted \cite{TANAKA}. For the case of weak homotopies we can begin by the fact that, clearly, $F \simeq_w F$ since we can take the homotopy $H \colon \CC \rightarrow \Pa \Dc$ given by taking every object $c$ to the constant path of length $0$ with object $F(c)$ and a morphism $f$ to the correspondent natural transformation associated to it.

    Now suppose that $F \simeq_w G$ with a homotopy $H \colon \CC \rightarrow \Pa \Dc$ we define $H' \colon \CC \rightarrow \Pa \Dc$ as $H'(c)$ the opposite path. 

    Finally suppose that $F \simeq_w G $ and $G \simeq_w P$ with homotopies $H_1$ and $H_2$. We define a homotopy $H \colon \CC \rightarrow \Pa \Dc$ as $H(c)$ as the composition of the path $H_1$ and $H_2$.
\end{proof}
    
\end{comment}

\begin{defi}
    Two small categories $\CC$ and $\Dc$ are {\em strong} (respectively {\em weak}) homotopy equivalent or {\em strong} (respectively {\em weak}) homotopic  if there are exist functors $F \colon \CC \rightarrow \Dc$ and $G \colon \Dc \rightarrow \CC$ such that $G \circ F \simeq \id_{\CC}$ and $F \circ G \simeq \id_{\Dc}$ (respectively $G \circ F \simeq_w \id_{\CC}$ and $F \circ G \simeq_w \id_{\Dc}$). The small category $\CC$  is {\em contractible} (respectively {\em weak contractible}) if it is strong homotopic (respectively weak homotopic) to the trivial category $\bullet$ with only one object and one morphism.
\end{defi}

\begin{prop}\label{prop:Strong_Homotopic_entails_weak_Homotpic}
    \begin{enumerate}
       \item If two functors are strong homotopic then they are also weak homotopic.
       \item As a consequence, if two categories are homotopic then they are also weak homotopic.
    \end{enumerate}   
\end{prop}

\begin{proof}
    Suppose that $F,G \colon \CC \rightarrow \Dc$ are strong homotopic and $H \colon \CC \times \mathbb{I}_n \rightarrow \Dc$ is a homotopy between then. We define the weak homotopy $H' \colon \CC \rightarrow \Pa \Dc$ defined as $H'(c)=H(c,-)$.
\end{proof}

\begin{comment}
\begin{corollary}
    If two Categories are homotopic then they are also weak homotopic.
\end{corollary}   
\end{comment}

The converse claim does not hold in general, as the following example shows.

\begin{example}\label{Zig-zag_infinity}
    Let us consider the  category $\mathbb{I}$ generated by the diagram:
    $$
    % https://tikzcd.yichuanshen.de/#N4Igdg9gJgpgziAXAbVABwnAlgFyxMJZABgBpiBdUkANwEMAbAVxiRGJAF9T1Nd9CKAIzkqtRizZCuPEBmx4CRAEyjq9Zq0QhlM3goFEAzGvGa2RvXL6LByACymNk7QDp3XMTCgBzeEVAAMwAnCABbJDIQHAgkIW4g0IjEVWjYxHjZEPCkVJikIwSQbOTHNILOCk4gA
\begin{tikzcd}
0 \arrow[r] & 1 & 2 \arrow[l] \arrow[r] & 3 & ... \arrow[l]
\end{tikzcd}$$
   We will show that $\mathbb{I}$ is not strong contractible but is weak contractible. 
   
   First, we show that $\mathbb{I}$ is weak contractible. A functor $F \colon \mathbb{I} \rightarrow \bullet$ takes every object to the object of $\bullet$ and every morphism to the identity. Consider the functor $0 \colon \bullet \rightarrow \mathbb{I}$ that takes the object of $\bullet$ to the object $0$. We have that $F \circ 0= \id_{\bullet}$. There is a weak homotopy between $\id_{\mathbb{I}}$ and $ 0 \circ F$ given by taking every object $n \in \mathbb{I}$ to the path of length $n$ from $0$ to $n$.

    Second, we show that $\mathbb{I}$ is not strong contractible. Suppose without loss of generality that there was a strong homotopy $H \colon \mathbb{I} \times \mathbb{I}_m \rightarrow \mathbb{I}$ such that $H(-,0)$ is the constant functor that takes everything to $0$ and $H(-,m)=\id_{\mathbb{I}}$. In that case for every object $n$ we would have a path of length $m$ from $0$ to $n$ but that is impossible if $n>m$.
\end{example}

Nevertheless, the example suggests a sufficient condition for weak homotopies being strong homotopies: finiteness of the domain. 

\begin{prop}\label{weak_Strong_Finite}
    Let $F,G \colon \CC \rightarrow \Dc$ be two weak homotopic functors between small categories. If $\CC$ is finite (the set of morphisms is finite), then $F$ and $G$ are strong homotopic.
\end{prop}

\begin{proof}
    Suppose that we have a weak homotopy $H \colon \CC \rightarrow \Pa \Ba$ such that $\pi_0 \circ H=F$ and $\pi_1 \circ H=G$. We have that since $\CC$ have a finite number of objects and morphism there is maximum lenght $m$ for the paths $H(c)$ and the intermediate path that appears in the morphism between them. We can then define a strong homotopy by taking $H^* \colon \CC \times \mathbb{I}_m \rightarrow \Dc$ to be a functor defined by taking some adequate equivalent paths for each object $c$ in $\CC$.
\end{proof}

\subsection{Notions of fibrations}

In order to  introduce categorical analogues of fibrations, we adapt the classical lifting properties from topology to our context. This yields two notions: strong fibrations, modeled on strong homotopies, and weak fibrations, based on weak homotopies.

\begin{defi}\label{Strong_Homotopy}
 Let $P\colon \E \to \Ba$ be a functor between small categories. \begin{enumerate}
     \item The functor $P$ is a {\em (finite) strong fibration} if for any (finite) small category $\CC$ and commutative diagram:
    \begin{equation*}
\begin{tikzcd}
\CC \arrow[d,"i_0"'] \arrow[r,"G"] & \E  \arrow[d,"P"] \\
\CC\times \mathbb{I}_m \arrow[r,"H"']                   & \Ba           
\end{tikzcd}
\end{equation*}
where $i_0(c)=(c,0)$ and $H$ is a strong homotopy, there exists a strong homotopy $\wtilde H \colon \CC\times \mathbb{I}_m \to\E$ such that the following diagram is commutative:
\begin{equation*}
\begin{tikzcd}
\CC \arrow[d,"i_0"'] \arrow[r,"G"] & \E  \arrow[d,"P"] \\
\CC\times \mathbb{I}_m \arrow[ur,"\wtilde H",dashed]\arrow[r,"H"']                   & \Ba           
\end{tikzcd}
\end{equation*}
\item The functor $P$ is a {\em (finite) weak fibration} if for every weak homotopy $H \colon \CC \rightarrow \mathrm{P} \Ba$ (where $\CC$ is finite) and functor $F \co \CC \rightarrow \E$ such that $\pi_0 \circ H=G$ there exists a weak homotopy $\wtilde{H}: \CC \rightarrow \mathrm{P}\E$ such that the following diagram is commutative:
            $$
% https://tikzcd.yichuanshen.de/#N4Igdg9gJgpgziAXAbVABwnAlgFyxMJZABgBoBGAXVJADcBDAGwFcYkQAdDgYW5AF9S6TLnyEU5CtTpNW7LgFt6OABYAnBcAAK-AARcAQvQFCQGbHgJFJxaQxZtEnDktUbt-LgFETwi2KIAJlJbGns5J29fMxFLcWRgqjDZR2cjAWkYKABzeCJQADM1CAUkMhAcCCRJGQd2AAloopLqmkqkABZkusiONCwAfWIm4tLEcvbEYNqI5xVlYHr+EZaptqrEGvDUxWV1TR1dLRWxgGZ1zu7Z48FC0aRpyfOZnY4AI3o1YC5+-iGQGiMehvGCMLSxAJONRYbIqHAnMoXRDPEFgKBIABs5W28g48xwwAAYst+JR+EA
\begin{tikzcd}
                                                                              & \mathrm{P}\E \arrow[d, "\mathrm{P} P"] \arrow[r, "\bar{\pi}_0"'] & \E \arrow[d, "P"] \\
\CC \arrow[r, "H"] \arrow[ru, "\wtilde{H}"] \arrow[rru, "{F}", bend left=60] & \mathrm{P} \Ba \arrow[r, "\pi_0"]                                & \Ba              
\end{tikzcd}
    $$
Moreover, $P\colon \E \to \Ba$ {\em preserves length} if there exists a weak homotopy $\wtilde{H}: \CC \rightarrow \mathrm{P}\E$ verifying the additional condition:  if $H(c) \colon \mathbb{I}_m \rightarrow \Ba$ is of lenght $m$, then $\wtilde{H}(c)$ is isomorphic to a path $I \colon \mathbb{I}_m \rightarrow \E$ of legth $m$. Under this circumstances we say that the functor functor $P$ is a {\em length-preserving weak fibration}. 
 \end{enumerate}
 In both cases we say that $\wH$ is a {\em lift} of $H$.
\end{defi}

\begin{remark}
    A functor that is a fibration and an op-fibration is a strong fibration (see \cite[Proposition 2.7]{Cohomologysvarcgenus}). Nevertheless, the converse is false (see \cite[Proposition 3.1]{CMM}).
\end{remark}

\begin{teo}\label{under_length_preserving_strong_weak_equivalence}
    Let $P \colon \E \rightarrow \Ba$ be a functor. \begin{enumerate}
        \item If the functor $P \colon \E \rightarrow \Ba$ is a length-preserving weak fibration, then it is a strong fibration.
        \item The functor $P \colon \E \rightarrow \Ba$ is a finite strong fibration iff it is a finite length-preserving weak fibration.
    \end{enumerate}
\end{teo}

\begin{proof}
We begin by proving (1). Let $H \colon \CC \times \mathbb{I}_m \rightarrow \Ba$ be a strong homotopy and let $F \colon \CC \rightarrow \E $ be a functor such that $P \circ F=H(c,0)$. The strong homotopy $H \colon \CC \times \mathbb{I}_m \rightarrow \Ba$ induces a functor $H^* \colon \CC \rightarrow \Pa\Ba$ given by $H^*(c)= H(c,-) \colon \mathbb{I}_m \rightarrow \Ba$. Observe that all paths have length $m$. Moreover, $\pi_0 \circ H^*= P \circ G$ since $\pi_0 \circ H^*(c)=H^*(c)(0)=H(c,0)$. The weak fibration condition guarantees that we can lift $H^*$ to a functor $\wtilde{H}^* \colon \CC \rightarrow \Pa \E$. Furthermore, this lift is length-preserving and since all paths have length $m$, we can define a lift for $H$ given by $\wtilde{H} \colon \CC  \times \mathbb{I}_m \rightarrow \E$ given as follows: $\wtilde{H}(c,i)=\wtilde{H}^*(c)(i)$ and for any morphism $f \colon c \rightarrow d$ we have that $\wtilde{H}(f,1_i)=\alpha_f(i)$ where $\alpha_f$ is the natural transformation appearing in $\wtilde{H}^*(f)=(F_f,\alpha_f)$.

We prove (2). First, we prove that a finite strong fibration is a finite weak fibration. Let $\CC$ be a finite small category, let $H\colon\CC \to \Pa \Ba$ be a weak homotopy, and let $F\colon \CC\to \E$ be a functor such that $P\circ F=\pi_0\circ H$. We must construct $\widetilde H\colon \CC\to \Pa \E$ with
\[
\Pa P\circ \widetilde H = H \qquad\text{and}\qquad \pi_0\circ \widetilde H = F,
\]
and such that $\widetilde H(c)$ has the same length as $H(c)$ for every onject $c$ in $\CC$.

For each object $c$ in $\CC$, we write $H(c)$ as $H_c\colon \mathbb{I}_{m(c)}\to \Ba$ with $H_c(0)=(P\circ F)(c)$. Since $\CC$ is finite, the finite set $\{m(c)\mid c\in C\}$ admits a maximum; denote it by
\[
M \;=\; \max_{c\in C} m(c).
\]
For each $c$ choose the endpoint–preserving inclusion $\iota_c:\mathbb{I}_{m(c)}\hookrightarrow \mathbb{I}_M$ (obtained by padding $H_c$ by repeating endpoints when needed). Define the \emph{padded} path $H_c^\ast \colon = H_c\circ \iota_c \colon \mathbb{I}_M\to \Ba$.

We now define a functor $H^\ast \colon \CC \times \mathbb{I}_M \longrightarrow \Ba$
by setting on objectsas $H^\ast(c,i) \;=\; H_c^\ast(i)$ and on morphisms as follows. Morphisms in $\CC\times \mathbb{I}_M$ are generated by those of the form $(\mathrm{id}_c,\sigma)$ with $\sigma$ a morphism of $\mathbb{I}_M$, and those of the form $(f,\mathrm{id}_i)$ with $f\colon c\to d$ in $\CC$.
\begin{itemize}
\item For $(\mathrm{id}_c,\sigma)$ we put $H^\ast(\mathrm{id}_c,\sigma):=H_c^\ast(\sigma)$, i.e.\ the corresponding arrow along the padded path $H_c^\ast$.
\item Let $f\colon c\to d$ in $\CC$. Since $H$ is a functor to $\Pa \Ba$, the arrow $H(f)$ is represented by a pair $(F_f,\alpha_f)$ with $F_f\colon \mathbb{I}_{m(c)}\to \mathbb{I}_{m(d)}$ and a natural transformation $\alpha_f\colon H_c\Rightarrow H_d\circ F_f$.
Using the inclusions $\iota_c,\iota_d$, define the \emph{extended reparametrization} $\overline F_f:I_M\to I_M$ by the relation
\[
\iota_d\circ F_f \;=\; \overline F_f\circ \iota_c .
\]
Then post-compose $\alpha_f$ with $\iota_c$ to obtain a natural transformation
\[
\alpha_f^\ast \;:=\; \alpha_f\ast \iota_c \;:\; H_c^\ast \Longrightarrow H_d^\ast \circ \overline F_f .
\]
We define $H^\ast(f,\mathrm{id}_i)$ to be the component $(\alpha_f^\ast)_i : H_c^\ast(i)\to H_d^\ast(\overline F_f(i))$.
\end{itemize}
These prescriptions respect the relations in $\CC\times \mathbb{I}_M$ and are compatible with compositions (the latter follows from functoriality of $H$ in $\Pa \Ba$ together with the defining relations in the localized path category). Hence $H^\ast$ is a well-defined functor. Moreover, by construction
\[
H^\ast(-,0) \;=\; P\circ F \;:\; \CC \longrightarrow \Ba .
\]

Since $P$ is a strong fibration (for finite domains), there exists a lift
\[
\widetilde H^\ast \;\colon \; \CC\times \mathbb{I}_M \longrightarrow \E
\]
such that $P\circ \widetilde H^\ast = H^\ast$ and $\widetilde H^\ast(-,0)=F$. Define
\[
\widetilde H(c) \;:=\; \widetilde H^\ast(c,-) \;:\; \mathbb{I}_M \longrightarrow \E .
\]
This gives a functor $\widetilde H\colon \CC\to \Pa \E$ satisfying $\pi_0\circ \widetilde H = F$ and $\Pa P\circ \widetilde H = H$ \emph{in $PB$}: indeed, $\Pa P\circ \widetilde H(c)=\Pa\circ \widetilde H^\ast(c,-)=H^\ast(c,-)$, which is the padded version of $H(c)$ and hence equal to $H(c)$ in the localization $\Pa \Ba$. Finally, since $\mathbb{I}_M$ has constant length $M$, the objectwise lifts $\widetilde H(c)=\widetilde H^\ast(c,-)$ have length $M$; but in $\Pa \E$ each $\widetilde H(c)$ is isomorphic to a path of length $m(c)$, showing that the weak lift preserves length.

We prove the converse. Let $H \colon \CC \times \mathbb{I}_m \rightarrow \Ba$ be a strong homotopy and let $F \colon \CC \rightarrow \E $ be a functor such that $P \circ F=H(c,0)$. The strong homotopy $H \colon \CC \times \mathbb{I}_m \rightarrow \Ba$ induces a functor $H^* \colon \CC \rightarrow \Pa\Ba$ given by $H^*(c)= H(c,-) \colon \mathbb{I}_m \rightarrow \Ba$. Moreover, $\pi_0 \circ H^*= P \circ G$ since $\pi_0 \circ H^*(c)=H^*(c)(0)=H(c,0)$. The weak fibration condition guarantees that we can lift $H^*$ to a functor $\wtilde{H}^* \colon \CC \rightarrow \Pa \E$. Since $\CC$ is finite, there is an $M$ such that all paths are at most of length $M$. Additionally, we can extend the paths shorter than $M$ to length $M$ by the constant path at the end point. We define a lift for $H$ given by $\wtilde{H} \colon \CC  \times \mathbb{I}_M \rightarrow \E$ given as follows: $\wtilde{H}(c,i)=\wtilde{H}^*(c)(i)$ and for any morphism $f \colon c \rightarrow d$ we have that $\wtilde{H}(f,1_i)=\alpha_f(i)$ where $\alpha_f$ is the natural transformation appearing in $\wtilde{H}^*(f)=(F_f,\alpha_f)$. 
\end{proof}

\subsection{First properties of fibrations}
We present some fundamental properties of fibrations, such as behavior under pullbacks and composition, which will be crucial for later applications.

\begin{prop}\label{Pullback}
\begin{enumerate}
    \item The pullback of a (finite) (length-preserving) weak fibration is a (finite) (length-preserving) weak fibration. 
    \item The pullback of a (finite) strong fibration is a (finite) strong fibration.
\end{enumerate}
\end{prop}

\begin{proof}
(1). We prove that the pullback of a length-preserving weak fibration is a length-preserving weak fibration and the other proofs are analogous. Consider the following pullback of a weak fibration $P \colon E \rightarrow \Ba$:
$$
% https://tikzcd.yichuanshen.de/#N4Igdg9gJgpgziAXAbVABwnAlgFyxMJZABgBpiBdUkANwEMAbAVxiRAB12BRAchAF9S6TLnyEUZAIxVajFm04AhOn0HDseAkUnkZ9Zq0QduAoSAwax20tOr75RpXQEyYUAObwioAGYAnCABbJAAmahwIJABmOzlDEAAFU18A4MQdEAjo2IM2AHFkkH8gpDJMyPSch0TVM2K0sqzEMNlcozzVCn4gA
\begin{tikzcd}
\E' \arrow[d, "P'"] \arrow[r, "G'"] & \E \arrow[d, "P"] \\
\Ba' \arrow[r, "G"]                 & \Ba              
\end{tikzcd}
$$
We know that the category $\E'$ is the subcategory of $\Ba' \times \E$ given by objects $(b,e) \in \Ba' \times \E$ such that $G(b)=P(e)$. Moreover the universal property of the pullback allows us to view a functor $M \colon \CC\rightarrow \E'$ as given by $M(c)=(M_1(c),M_2(c))$ such that $G \circ M_1(C)= P \circ M_2(C)$. In particular, we can use this to understand that a path in $\E'$ is a suitable pair of paths in $\Ba'$ and $\E$.

Now suppose that we have a weak homotopy $H \colon \CC \rightarrow \mathrm{P}B'$ this give rises to a weak homotopy $H' \colon \CC \rightarrow \mathrm{P} B$ given by $H'= \mathrm{P}G \circ H$. This means that we have the following diagram: 
$$
% https://tikzcd.yichuanshen.de/#N4Igdg9gJgpgziAXAbVABwnAlgFyxMJZABgBoBGAXVJADcBDAGwFcYkQAdDgYW5AF9S6TLnyEU5CtTpNW7LgAV6AAi4AhegHIBQkBmx4CRScWkMWbRJw5LVHAKLbBwg2KIAmKTXNyriler0Oi6iRiiept6ylta2XPYC0jBQAObwRKAAZgBOEAC2SGQgOBBIkjIW7AASwSA5+UiexaWI5T4x-soKTrr1Ba00JUgAzFGVfjYqAOK1fY2DLQAsY76x0z1Zuf3LzSMrHZNds1uFC3sgAEYwYFBIALTDRe3V2jSM9FeMCiKG4iDZWBSAAscIl+EA
\begin{tikzcd}
                                                 & \Pa \E' \arrow[d, "\Pa P'"] \arrow[r, "\Pa G'"] & \Pa \E \arrow[d, "\Pa P"] \\
\CC \arrow[r, "H"] \arrow[rr, "H'"', bend right] & \Pa \Ba' \arrow[r, "\Pa G"]                     & \Pa \Ba                  
\end{tikzcd}
$$

Moreover if we have a functor $F \colon \CC \rightarrow E'$ we also have a functor $F'=G' \circ \hat{F}$. With this we obtain the following diagram that give us a complete picture of the situation:
$$
% https://tikzcd.yichuanshen.de/#N4Igdg9gJgpgziAXAbVABwnAlgFyxMJZABgBoAmAXVJADcBDAGwFcYkQAdDgYW5AF9S6TLnyEUARgrU6TVuy4AFegAIuAIXoByAUJAZseAkQDM0mgxZtEnDsrUdNu4YbFFypEzMvybG7c76IkbiyAAsnt5y1rZOgi6ixpKkxFFWCnaqXACiOvFBrknIZqkW0Rn2OYEGiaEeEmm+trnVwW4oEQ1l6X4c2QIyMFAA5vBEoABmAE4QALZIZCA4EEhSsj0gABIgNIz0AEYwjIptSSBYYNiwgdNzqzTLSB4gjBcxcBCvUDsvB0cnhXEICmWGGAAscDcZvNEGtHogTPlbjDnvCwkjoUgzEsVogIusmgBxKF3RAAVgeuIA7BjSQA2SlIAActJhFJx9wJMSUqkUeT0yKQDI5iHIrOZjLx3SaihJMKpkuxe0Ox1OQIuVzY0pifJ+cDBWAmkMQxHF5Mlwp83MyKkJ-MmmMQCpFTO17DtP2V-zV7A1WGuZsW8OdVvYADFPX9VYD2IwYEaBvwgA
\begin{tikzcd}
                                                & \Pa \E' \arrow[rd] \arrow[dd, "\Pa P'"] \arrow[rr, "\Pa G'"] &                                                               & \Pa \E \arrow[rd] \arrow[dd] &                    \\
                                                &                                                              & \E' \arrow[dd, "P'" description] \arrow[rr, "G'" description] &                              & \E \arrow[dd, "P"] \\
\CC \arrow[r, "H" description] \arrow[rru, "F"] & \Pa \Ba' \arrow[rr] \arrow[rd]                               &                                                               & \Pa \Ba \arrow[rd]           &                    \\
                                                &                                                              & \Ba' \arrow[rr, "G"]                                          &                              & \Ba               
\end{tikzcd}$$
where some morphism are not explicitly labeled for readability.

Therefore we can lift the weak homotopy $H'$ to $\widetilde{H}'$ but this means that we can define a suitable lift $\wH$ of the original $H$ as:
$$\wH(c)=(H(c),\wH'(c)).$$

(2) in the finiteness setting follows from (1) and Theorem \ref{under_length_preserving_strong_weak_equivalence} (2).  It remais to prove it in the non-finite context. Take the pullback $P'\colon \E' \rightarrow \Ba'$ of a strong fibration $P \colon \E \rightarrow \Ba$ as in the previous case. If we have the following commutative diagram:
$$
% https://tikzcd.yichuanshen.de/#N4Igdg9gJgpgziAXAbVABwnAlgFyxMJZABgBpiBdUkANwEMAbAVxiRAB12BhLkAX1LpMufIRRkAjFVqMWbTjwAEnPAFt4y9qro4AFgCN9wAJJ8A+qv6CQGbHgJEJpKdXrNWiDuwBCdAORWQnaijuTSbnKenACiAQJBIg4oAEzO4bIeXr6BNsL2YsiplK4Z8uzR-NIwUADm8ESgAGYAThCWiADM1DgQSKky7mwACnHWLW1ITiA9fSWDngASOePtZNO9iBLxICtIazOdc5EgAGLLre39BwAsR5kA4ucTiACs3Ru3A8dDT+1d60g3l8HnEKHwgA
\begin{tikzcd}
\CC \arrow[d] \arrow[r, "F"]           & \E' \arrow[d, "P'"] \arrow[r, "G'"] & \E \arrow[d, "P"] \\
\CC \times \mathbb{I}_m \arrow[r, "H"] & \Ba' \arrow[r, "G"]                 & \Ba              
\end{tikzcd}
$$
we can define a homotopy $H'=G \circ H$ and we can take the lift of $H'$ and $G'\circ F$ that we denoted by $\wH' \colon \CC \times \mathbb{I}_m \rightarrow \E$. Now we define the homotopy lift of the original homotopy $H$ as $\wH \colon \CC \times \mathbb{I}_m \rightarrow \E'$ as $\wH(c,t)=(H(c,t),\wH(c,t)$.
\end{proof}

\begin{prop}\label{prop:trivial_fibration}
Let $\bullet$ be the category with only one object and one morphism, namely the identity. The functor $P \colon \Ba \rightarrow \bullet$  is a lenght-preserving weak (and therefore strong) fibration.
\end{prop}

\begin{proof}
Let $P \colon \Ba \rightarrow \bullet$ be a functor. Take $H \colon \CC \rightarrow \Pa \bullet$ and $F \colon \CC \rightarrow \Ba$. The condition that $F \circ P= \pi_0 \circ H$ is trivially satisfied by the fact that $\bullet$ is the terminal category. We need $\wH \colon \CC \rightarrow \Pa \Ba$ such that $\Pa P \circ \wH = H$. We can define this taking as $\wH(c)$ the constant path of the same length as $H(c)$ given by $F(c)$ and $\wH(f)$ as the morphism given by $F(f)$.   
\end{proof}

\begin{corollary}\label{Projection}
    The projections $p_1 \colon \CC \times \Dc \rightarrow \CC$ and $p_2 \colon \CC \times \Dc \rightarrow \Dc$ are weak and strong fibrations.
\end{corollary}
\begin{proof}
    We prove the result for $p_1 \colon \CC \times \Dc \rightarrow \CC$ and the other is analogous. 
Consider the terminal category $\bullet$. By Proposition \ref{prop:trivial_fibration} the functor
$P:D \to \bullet$ is a weak and strong fibration. Now observe that the product
$C\times D$ together with the projection $p_{1}:C\times D \to C$
arises as the pullback of the diagram
\[
C \xrightarrow{\,  \,} \bullet \xleftarrow{\, P \,} D.
\]
Now the result follows from Proposition \ref{Pullback}.    
\end{proof}

\begin{prop}\label{Composition_Fibration}
    Let $P_1 \colon \E_2 \rightarrow \E_1$ and $P_2 \colon \E_1 \rightarrow \Ba$ be two weak fibrations. The composition $P_2\circ P_1$ is also a weak fibration.
\end{prop}
\begin{proof}
   Suppose that we have a suitable pair $F \colon \CC \rightarrow \E_2$ and $H \colon \CC \rightarrow \Pa \Ba$. This defines a diagram:
   $$
   % https://tikzcd.yichuanshen.de/#N4Igdg9gJgpgziAXAbVABwnAlgFyxMJZABgBoAmAXVJADcBDAGwFcYkQAdDgYW5AF9S6TLnyEUARgrU6TVuy4BbejgAWAJ0XAACvwAEXAEL0BQkBmx4CRKRJkMWbRJw7K1mnfy4BRU8MtiROSkdjQO8s4+fuYiVuLIwVRhck4uxtEWotaSpMT2KQquKhpauj4A+uQZsYEowXnJjoXelQIyMFAA5vBEoABm6hCKSGQgOBBIUrJNzgAS0QNDkzTjSAAsgv2Dw4jBYxOIU+GpSsUeunra5RIL20gAzCsHa40RIFc3myCLO3uriI8QIx6AAjGCMbQ1bIgdRYTqqHC3JaIACsTyQe2OhTcJU8lyqXx+SAAbOiAa9UlcCWYiaiycTCXdEKN-qSQGCwFAHqMsc4AGJtfhAA
\begin{tikzcd}
                                                & \mathrm{P}\E_2 \arrow[d, "\mathrm{P} P2"] \arrow[r] & \E_2 \arrow[d, "P_2"] \\
                                                & \mathrm{P}\E \arrow[d, "\mathrm{P} P_1"] \arrow[r]  & \E \arrow[d, "P_1"]   \\
\CC \arrow[r, "H"] \arrow[rruu, "F", bend left] & \mathrm{P} \Ba \arrow[r]                            & \Ba                  
\end{tikzcd}
   $$
   since $\Pa$ is functorial. Using the fact that $P_1$ is a fibration there is a lift $\wH \colon \CC \rightarrow \Pa \E_1$ for $H$ and $P_2 \circ F$. Now using that $P_2$ is also fibration we obtain a lift $\wH^* \colon \CC \rightarrow \Pa \E_2$ for $\wH$ and $F$. This also a lift for $H$ and $F$ since $\wH^* \circ \Pa (P_1 \circ P_2)= \wH^* \circ \Pa (P_1) \circ \Pa (P_2)= \wH \circ \Pa(P_2)= H$ and $\pi_0 \circ \wH^*=F$.
\end{proof}

\section{The fibrant replacement of a functor and Svarc genus}\label{sec:fibrant_replacement_svarc}

\subsection{The fibrant replacement of a functor}
A central technique in homotopical settings is to replace a given object by a fibrant one. Here we adapt this idea to functors between small categories, constructing a factorization that provides a weak homotopy equivalence followed by a fibration.

\begin{teo}[Fibrant replacement for a functor]
\label{functor_factorisation_through_fibration}
    Let $F\co \CC \rightarrow \Dc$  be a functor. Then $F=F_2\circ F_1\co \CC \to E_F \to \Dc$ where $F_1\co \CC\rightarrow \mathrm{E}_F$ is a weak homotopy equivalence and $F_2\co \mathrm{E}_F \rightarrow \Dc$ is a length-preserving weak fibration (therefore a strong fibration).
\end{teo}

\begin{proof}
    We define the category $\mathrm{E}_F$ as the pullback:
    $$
    % https://tikzcd.yichuanshen.de/#N4Igdg9gJgpgziAXAbVABwnAlgFyxMJZABgBoBGAXVJADcBDAGwFcYkQAdDgYW5AF9S6TLnyEU5CtTpNW7LgBEAxgKEgM2PASKTi0hizaJOHALb0cACwBOp4AAV+AAkUrBwzWKJk9NA3OMucytbYABRfgB9ADEBaRgoAHN4IlAAM2sIUyQyEBwIJEkZQ3ZY9xAMrKQAJhp8wr9ZIxMsmET6SPJVdMzsxABmOoLEavLKvsG84eJ+Sn4gA
\begin{tikzcd}
\mathrm{E}_F \arrow[r] \arrow[d] & \mathrm{P} \Dc \arrow[d, "\pi_0"] \\
\CC \arrow[r, "F"]               & \Dc    \end{tikzcd}
    $$
    That is, $\mathrm{E}_F$ is the category whose objects are pairs $(c,I)$ where $c$ is an object in $\CC$ and $I$ is a path in $\Dc$ such that its initial object is $F(c)$. A morphism is given by an arrow $f \colon c \rightarrow c'$ in $\CC$ and a morphism of paths between $I$ and $I'$ such that the morphism connecting their initial objects is $F(f)$.

We now define the functors $F_1:\CC \rightarrow \mathrm{E}_F$ and $F_2: \mathrm{E}_F \rightarrow \Dc$. We define $F_2$ as the functor that sends an object $(c,I)$ to the final object of $I$ and a morphism to the corresponding connecting morphism between the final objects. The functor $F_1$ is defined by $F_1(c)=(c,F(c))$, where, by an abuse of notation, $F(c)$ denotes the constant path of length $0$ associated to the object $F(c)$; it takes a morphism $f \colon c \rightarrow c'$ to the pair consisting of the morphism $f$ and the morphism between the constant paths given by $F(f)$.

    In order to prove that $F_1$ is a weak homotopy equivalence, we must define a weak homotopy inverse. Let $G \colon \mathrm{E}F \rightarrow \CC$ be the functor given by the universal property of the pullback, i.e., it projects to the first component. It is easy to see that $G \circ F_1=\id{\CC}$, so they are also weakly homotopic. On the other hand, $F_1 \circ G$ is the functor that sends an object $(c,I)$—where $I\colon \mathbb{I}_m \rightarrow \CC$ is a path such that $I(0)=F(c)$—to $(c,F(c))$. To obtain a homotopy between $F_1 \circ G$ and the identity, we must assign, in a functorial way, a path from $(c,F(c))$ to $(c,I)$. This construction, however, is straightforward. Suppose that $I$ is a path of the form:
    $$
    % https://tikzcd.yichuanshen.de/#N4Igdg9gJgpgziAXAbVABwnAlgFyxMJZABgBpiBdUkANwEMAbAVxiRADEAKAYQEoQAvqXSZc+QigCM5KrUYs2ASQD6kwcJAZseAkQBMM6vWatEIADrmGUCDgRCR28UQDMhuSaXLgYALSSBdUcxXRQAFndjBTMVQgFZGCgAc3giUAAzACcIAFskMhAcCCRpD2iQLGViIJAs3KQDQuLEUoY6ACMYBgAFUR0JEEysJIALHBq6vMRGoqQXB1rsqYimuaN5UwrvPwCQajbOnr7nMyHR8fiBIA
\begin{tikzcd}
F(c) \arrow[r, "i_0"] & I_1 & \ldots \arrow[l] \arrow[r] & I_{m-1} & I_m \arrow[l, "i_{m-1}"']
\end{tikzcd}
$$
    we can define a path of paths between $I$ and the constant path $F(c)$ using the following diagram:
    $$
% https://tikzcd.yichuanshen.de/#N4Igdg9gJgpgziAXAbVABwnAlgFyxMJZABgBoAWAXVJADcBDAGwFcYkQAxACgGMBKEAF9S6TLnyEUARgrU6TVuwCSXKQOGjseAkQBMsmgxZtEIADpnGUCDgQaQGLRKIBmA-OPKuAWwC0aoREHMW1JZHJ3I0VTFW91IMdxHRQyKTkok05eeM0ksJk0wwVMlQD7RNC9UkKPaPNLa1tA3MqUNxqMrzKEkOcUCI7irpzgp2SSUl10odMLKxs7HrGwsmJpz1NufmbRvKIZNaKNrO3y3vH9Q9rMucbFlr7kNyvOzeydiseIl5mTkc-xjIpkc6rcFh9zmF9MDruwwU0zstXJN1qCGuDEXt+iiQTd0QilljwqQXKiSj5-P9IUQyKTcewtlSkdISWThhDmch9HTYbN8fddq0nqz6TEKWU5DAoABzeBEUAAMwAThBvEgyCAcBAkDJeSAsAB9Yg7ZWqpD6TXaxBSeymtWIC1apAuW0q+0RS3O0X6g3APw2k1unW6p3Wi2vEBSQNm61uT1h71RmiMegAIxgjAACtTTEqsNKABY4aP2qRx0NSD0RqOumNSACsNArVd+SZAKfTWZzIDzheLtftjfjADZvYbjQOkAB2JtW0fttMZ7PMnv5osgRMl6ezpAADk3k8QAE4d4h93q2x2l93e+vD7oQ1bdOHfuOtw7y0-dPfP+byIepEfJAh2rd8pEdOcxyNMDf0QKcAI9UNzwjN8AI1UMQNbMD51DYcAJneN4KCO0dXPJCAJPeMj3vdCnxbY4a2IoNrVo81dVA5NFy7Fdb37Ji6xwr98NPXQXX40syLoiiROIQ9EKfDUUN9MB-EEd9YN0ei6kYxVmIg80xN0mNWIddisIAkSXwYsCh1DSilIncSDNPKR4MoQQgA
\begin{tikzcd}
F(c) \arrow[r, "1"] \arrow[d, "1"]    & F(c) \arrow[d, "i_0"]      & \ldots \arrow[l, "1"'] \arrow[r, "1"] \arrow[d] & F(c) \arrow[d, "i_0"]      & F(c) \arrow[l, "1"'] \arrow[d, "i_0"]          \\
F(c) \arrow[r, "i_0"]                 & I(1)                       & \ldots \arrow[l, "1"'] \arrow[r, "1"]           & I(1)                       & I(1) \arrow[l, "1"']                           \\
\ldots \arrow[u, "1"] \arrow[d, "1"'] & \ldots \arrow[u] \arrow[d] & \ldots \arrow[u] \arrow[d]                      & \ldots \arrow[u] \arrow[d] & \ldots \arrow[u] \arrow[d]                     \\
F(c) \arrow[r, "i_0"]                 & I(1)                       & \ldots \arrow[l] \arrow[r] \arrow[u]            & I(m-1)                     & I(m-1) \arrow[l, "1"]                          \\
F(c) \arrow[r, "i_0"] \arrow[u, "1"]  & I(1) \arrow[u, "1"]        & \ldots \arrow[l] \arrow[r] \arrow[u]            & I(m-1) \arrow[u, "1"]      & I(m) \arrow[l, "i_{m-1}"] \arrow[u, "i_{n-1}"]
\end{tikzcd}
    $$
    Moreover, we can show that this construction is functorial. Consider a morphism $f \colon c \rightarrow c'$ and a map of paths $(L,\alpha) \colon I \rightarrow J$, where $I \colon \mathbb{I}_m \rightarrow \Dc$ and $J \colon \mathbb{I}_n \rightarrow \Dc$ with $I(0)=F(c)$ and $J(0)=F(c')$. This map can be represented by the following commutative diagram:
    $$
    \adjustbox{scale=0.80,center}{
   % https://tikzcd.yichuanshen.de/#N4Igdg9gJgpgziAXAbVABwnAlgFyxMJZARgBoAGAXVJADcBDAGwFcYkQAxACgGMBKEAF9S6TLnyEUAJgrU6TVuwCSXYgOGjseAkQDMsmgxZtEIADpnGUCDgQaQGLRKIAWA-OPKuAWwC0aoREHMW1JZABWdyNFUxVvdSDHcR0UclJiOWiTTl4AcgTNZLCyDMMFbIApLgAZVT4C4KcU5BlSjxjzS2tbQMLQvXTM8vYq2r81BqT+lDc2rJGan3qhORgoAHN4IlAAMwAnCG8kNJAcCCRie33DpBlT88QyEEZ6ACMYRgAFEOdTPax1gALHC9EDXI6IO5nJC6K4HCFue4wmgvd5fH4pED-IEguE3RAnaGISLtbIWJhoQH0AD65AAvDsQCi3h9vk1JFiAcDQeCLjQiQA2MqeUzkxiUmnEJnPFno9nsbHcvEQqEPADsyuRSMQAA5hR0xRLqcBxoIefCkIiiQBOfVkswUqnUo6a4n8h4C11q91IT1BXmIb3anWu23B5lotlFBVc3GUQRAA
\begin{tikzcd}
                & F(c) \arrow[r] \arrow[ld, "\alpha_0=f"'] & I(1) \arrow[ld, "\alpha_1"'] & \ldots \arrow[l] \arrow[r] \arrow[ld] & I(m-1) \arrow[ld, "\alpha_{m-1}"] & I(m) \arrow[l] \arrow[ld, "\alpha_m"] \\
F(c') \arrow[r] & J(L(1))                                  & \ldots \arrow[l] \arrow[r]   & J(L(m-1))                             & J(L(m)) \arrow[l]                 &                                      
\end{tikzcd}}
    $$
    This can be arranged to made a morphism between paths of paths using the following construction represented by the commutative diagram:
    $$
    \adjustbox{width=\textwidth}{
\begin{tikzcd}
                                      & F(c) \arrow[r] \arrow[ld] \arrow[dd, dashed]            & F(c) \arrow[ld] \arrow[dd, dashed]               & \ldots \arrow[l] \arrow[r] \arrow[ld] \arrow[dd, dashed] & F(c) \arrow[ld] \arrow[dd, dashed]         & F(c) \arrow[l] \arrow[ld] \arrow[dd, dashed]      \\
F(c') \arrow[r] \arrow[dd]            & F(d) \arrow[dd, bend right]                             & F(d) \arrow[l] \arrow[dd, bend right]            & \ldots \arrow[l] \arrow[r] \arrow[dd, bend right]        & F(d) \arrow[dd, bend right]                &                                                   \\
                                      & F(c) \arrow[ld] \arrow[r]                               & I(1) \arrow[ld]                                  & \ldots \arrow[ld] \arrow[l] \arrow[r]                    & I(1) \arrow[ld]                            & I(1) \arrow[ld] \arrow[l]                         \\
F(c') \arrow[r]                       & J(L(1))                                                 & \ldots \arrow[l] \arrow[r]                       & J(L(1))                                                  & J(L(1)) \arrow[l]                          &                                                   \\
                                      & \dots \arrow[uu] \arrow[ld]                             &                                                  &                                                          &                                            & \dots \arrow[uu] \arrow[ld]                       \\
\dots \arrow[uu]                      &                                                         &                                                  &                                                          & \dots                                      &                                                   \\
                                      & F(c) \arrow[ld] \arrow[r] \arrow[uu] \arrow[dd, dashed] & I(1) \arrow[ld] \arrow[dd, dashed]               & \dots \arrow[ld] \arrow[r] \arrow[l]                     & I(m-1) \arrow[ld] \arrow[dd, dashed]       & I(m-1) \arrow[ld] \arrow[l] \arrow[uu] \arrow[dd] \\
F(c') \arrow[uu] \arrow[dd] \arrow[r] & J(L(1)) \arrow[dd, bend right]                          & \dots \arrow[l] \arrow[r] \arrow[dd, bend right] & J(L(m-1)) \arrow[dd, bend right]                         & J(L(m-1)) \arrow[l] \arrow[dd, bend right] &                                                   \\
                                      & F(c) \arrow[r] \arrow[ld]                               & I(1) \arrow[r]                                   & \dots                                                    & I(m-1) \arrow[r] \arrow[l]                 & I(m) \arrow[ld]                                   \\
F(c') \arrow[r]                       & J(L(1))                                                 & \dots \arrow[l] \arrow[r]                        & J(L(m-1)) \arrow[r]                                      & J(L(m))                                    &                                                  
\end{tikzcd}}
    $$
    This is well-defined in the category of paths, as one can easily verify that it respects weak equivalences. Therefore, we conclude that $\Ef$ is weak homotopy equivalent to $\CC$.

We now prove that $F_2$ is a weak fibration. Suppose we have a homotopy $H \colon \mathcal{B} \rightarrow \Pa \D$ and a functor $G \colon \Ba \rightarrow \Ef$ such that $G \circ F_2 = H \circ \pi_1$. We define a lifting $\wH \colon \Ba \rightarrow \Ef$ as follows. For each object $b$, we must construct a sequence of objects $(c_i, I_i)$ in $\Ef$ such that for each $i$, we have $I_i(m_i) = H(b, i)$, where $m_i$ is the length of $I_i$, along with a suitable zig-zag of morphisms between these paths.
text
We begin the construction with the path $I_0$, which is the second component of $G(b) = (c_b, I_b)$:
    $$
% https://tikzcd.yichuanshen.de/#N4Igdg9gJgpgziAXAbVABwnAlgFyxMJZABgBpiBdUkANwEMAbAVxiRAEkB9AIwApiAlCAC+pdJlz5CKAIzkqtRizZc+MoaPHY8BIgCZ51es1aIQAHXNQIOBJpAZtUogGZDikyp68AtjwEAvAASvNzkGgowUADm8ESgAGYAThA+SGQgOBBIMvbJqUgGmdmIuWIg+WmIRVlILsIUwkA
\begin{tikzcd}
I_b(0) \arrow[r] & I_b(1) & \dots \arrow[l] \arrow[r] & {I_b(m_b)=H(b,0)}
\end{tikzcd}$$
Now we can define $I_1$ and the corresponding morphism as follows:
$$
\adjustbox{width=\textwidth}{% https://tikzcd.yichuanshen.de/#N4Igdg9gJgpgziAXAbVABwnAlgFyxMJZABgBpiBdUkANwEMAbAVxiRAEkB9AIwApiAlCAC+pdJlz5CKAIzkqtRizZc+MoaPHY8BIgCZ51es1aIQAHXNQIOBJpAZtUogGZDikyp68AtjwEAvAASvNzkGmIOEjrSJKQyCsbKZqr8EVqSurLxiUqmHN7qIpGOmbEGCUZ5bJbWtsUZMa45VZ4p3n7c6VFOWcgALO5J+SFhgg09ZUSDlR7JIKPhE6VNKACsLXMjofHdK87rQ9Vmi+PCCjBQAObwRKAAZgBOED5IZCA4EEgy9k8vSAYPl9ED9In9XohAZ8kC5fs8IYMgUg1nD-ogAGzUaGIFFg+FITFIxAAdlRELcRIAHGSkABOLHA0l4tGUhl0mmICnYpkPfEgtk4jnvbH9Dn0okyYgc1kS0G8tEyOQSqUUYRAA
\begin{tikzcd}
I_b(0) \arrow[r] \arrow[d] & I_b(1) \arrow[d] & \dots \arrow[l] \arrow[r] & {I_b(m_b)=H(b,0)} \arrow[r] \arrow[d] & {H(b,0)} \arrow[d] \arrow[r] & {H(b,0)} \arrow[d] \\
I_b(0) \arrow[r]           & I_b(1)           & \dots \arrow[l] \arrow[r] & I_b(m_b)                              & {H(b,0)} \arrow[l] \arrow[r] & {H(b,1)}          
\end{tikzcd}}
$$
where each vertical morphism is the identity except the last one that it is $H(1_b,i_1)$.
For the second step we have:
$$\adjustbox{width=\textwidth}{
% https://tikzcd.yichuanshen.de/#N4Igdg9gJgpgziAXAbVABwnAlgFyxMJZABgBpiBdUkANwEMAbAVxiRAEkB9AIwApiAlCAC+pdJlz5CKAIzkqtRizZc+MoaPHY8BIgCZ51es1aIQAHXNQIOBJpAZtUogGZDikyp68Atjw1iDhI60sgALO7GymYAErzcpOoigY6SuihkMgpRphzegslaaaFyWUZKuaq8SfapIfqJ2RVslta2hUFO6chuZR7ReXx+3AFF9SgRfTlscQk1KcHOKACsjeWesfGkeqOdxUSrlOsDs9saCjBQAObwRKAAZgBOED5IZCA4EEgy9k8vSAYPl9ED9An9Xog3ECAb9nhCoZ8kGFYf9EKtoYgAGwoiEAdmoiKxOKQAA4CcDccTEGSMQBOKkyd6E+lguHfOQY5Gs1HownEKmY8nfKk0wkuKm0oWILkPNkgpnAmQ-CjCIA
\begin{tikzcd}
I_b(0) \arrow[r]           & I_b(1)           & \dots \arrow[l] & I_b(m_b) \arrow[l] \arrow[r]           & {H(b,1)}           & {H(b,2)} \arrow[l]           \\
I_b(0) \arrow[r] \arrow[u] & I_b(1) \arrow[u] & \dots \arrow[l] & I_b(m_b) \arrow[l] \arrow[r] \arrow[u] & {H(b,1)} \arrow[u] & {H(b,2)} \arrow[l] \arrow[u]
\end{tikzcd}
}
$$
We can repeat this until we get all the elements in $H(b,-)$. One can also check that this construction is funtorial.
\end{proof}

In the previous result it is necessary that the first is a weak homotopy equivalence and not a strong equivalence as the following example shows.

\begin{example}
    Suppose we have a connected category $\mathcal{C}$. Then, a constant functor $T: \bullet \to \mathcal{C}$ can be factorized as $\bullet \to \mathcal{E}_T \to \mathcal{C}$, where $\mathcal{E}_T$ is the category of paths in $\mathcal{C}$ with initial object $T(\bullet)$. The first functor $T_1$ maps $\bullet$ to $(\bullet, T(\bullet))$, and the second functor $T_2: \mathcal{E}_T \to \mathcal{C}$ sends a path with initial object $T(\bullet)$ to its end object.

    Now, consider the case where $\mathcal{C} = \mathbb{I}$, the category defined in Example \ref{Zig-zag_infinity}, and let $0: \bullet \to \mathbb{I}$ be the constant functor associated with the object $0$ in $\mathbb{I}$. In this case, there is no strong homotopy equivalence between the category $\bullet$ and $\mathcal{E}_0$. Indeed, if such a homotopy equivalence existed, there would be a positive integer $n$ such that for every path $I$ in $\mathbb{I}$, there would be a path of length $n$ connecting $I$ and the constant path $\widehat{0}$. However, this is impossible since there exist paths of arbitrary length.
\end{example}

\subsection{\v{S}varc genus and sectional category}
We recall from \cite{MAC-MOSQ-FUNCT} and \cite{TANAKA} the notions of geometric cover, sectional category and of (strong) Svarc genus. Let $\CC$ be a small category. A family of subcategories $\{\Uc_0,\dots,\Uc_n\}$ of $\CC$ is a {\em geometric cover} of $\CC$ if for every chain of composable morphisms $(f_1,\dots,f_m)$ in $\CC$ there is some $0\leq i \leq n$ such that every $f_j \in \mathrm{Mor}(\Uc_i)$ for each $1\leq j\leq m$.

\begin{defi}
Let $P\co \CC \rightarrow \Dc$ be a functor, its {\em  sectional category} denoted by $\sC (P)$ is  the least non-negative integer $n$ such that there is a geometric cover $\{\Uc_0, \ldots, \Uc_n \}$ of $\CC$ by subcategories that admit local sections, that is, such that for every $i \in \{0, \ldots, n\} $ there is a functor $s_i \colon \Uc_i \rightarrow \Dc$ which is a right inverse of $F$,  in the sense that $P \circ s_i = \iota_i$ ($P \circ s_i \simeq \iota_i$), where $\iota_i \colon \Uc_i \rightarrow \CC$ is the inclusion.   
\end{defi}

We also have a weaker notion, the Strong Svarc genus introduced in \cite{Cohomologysvarcgenus}.

\begin{defi}
Let $P\co \CC \rightarrow \Dc$ be a functor, its {\em strong homotopic} (or {\em strong Svarc genus}), denoted by $\sG (P)$, is  the least non-negative integer $n$ such that there is a geometric cover $\{\Uc_0, \ldots, \Uc_n \}$ of $\CC$ by subcategories that admit strong homotopic sections, that is, such that for every $i \in \{0, \ldots, n\} $ there is a functor $s_i \colon \Uc_i \rightarrow \Dc$ which is a strong homotopic right inverse of $F$,  in the sense that $P \circ s_i \simeq \iota_i$, where $\iota_i \colon \Uc_i \rightarrow \CC$ is the inclusion.   
\end{defi}

We introduce the novel concept of {\em weak homotopic sectional category} or {\em weak \v{S}varc genus}.

\begin{defi}
    Let $F\co \CC \rightarrow \Dc$ be a functor, its {\em weak homotopic sectional category} or {\em weak \v{S}varc genus}, denoted by $\sG_w (P)$, is the least positive integer $n$ such that there is a geometric cover $\{\Uc_0, \ldots, \Uc_n \}$ of $\CC$ by subcategories that admit local weak homotopic sections, i.e. such that for every $i \in \{0, \ldots, n\} $ there is a functor $s_i \colon \Uc_i \rightarrow \Dc$ that is a right weak homotopic inverse of $F$, in the sense that $P \circ s_i \simeq_w \iota_i$.
\end{defi}

\begin{prop}\label{prop:homotopic-sections_fibration_are_sections}
    Let $P\co\E \rightarrow \Ba$ be a functor. 
    \begin{enumerate}
        \item If $P\co\E \rightarrow \Ba$ is a strong fibration, then     $\sG_s(P)=\sC(P).$
        \item If  $P\co\E \rightarrow \Ba$ is a weak fibration, then 
    $\sG_w(P)=\sC(P).$
    \end{enumerate}
\end{prop}

%\q{He definido comandos para sC y sG. En mi opini\'on llamarle sG al g\'enero no tiene l\'ogica, pero no lo he cambiado}
\begin{proof}
    We begin by proving (1). It holds that $sG(P) \leq sC(P)$ since any local section is a homotopic section. To prove the converse inequality take any geometric cover $\{\Uc_0, \ldots, \Uc_n \}$ with local homotopic sections $s_i$. For each $i \in \{0,\ldots,n\}$ we have a strong homotopy $H \colon \Uc_i \times \mathbb{I}_m \rightarrow \Ba$ such that $H_0= \iota_i, H_m= P\circ s_i$ and the following diagram commutes:
    $$
    % https://tikzcd.yichuanshen.de/#N4Igdg9gJgpgziAXAbVABwnAlgFyxMJZABgBoBGAXVJADcBDAGwFcYkQBVAfSwAIAdfngC28Af2H0cACwBGs4AEkAvl2EhlpdJlz5CKchWp0mrdoIBC9DVpAZseAkTLFjDFm0SceN7Q71Ehq407mZeggCiGsYwUADm8ESgAGYAThDqiGQgOBBIhiYe7AASIDSM9LIwjAAKOo76IKlYcdI4viBpGUgATDS5SADMIaaeIHA+minpmcM5eYgFoWM1HV2ZffNI2cvsWGplIBVVtfUBXs2t7cqUykA
\begin{tikzcd}
U_i \arrow[r, "s_i"] \arrow[d, "i_m"']  & \E \arrow[d, "P"] \\
U_i \times \mathbb{I}_m \arrow[r, "H"'] & \Ba              
\end{tikzcd}
    $$
    Now by the homotopy lifting property there is a homotopy $\wtilde H  \colon \Uc_i \times \mathbb{I}_m \rightarrow \E$ such that $P \circ \wtilde H _0= H_0= \iota_i$, so $\wtilde H_0$ is a local section.

    We address (2). The same argument as before shows $sG_w(P) \leq sC(P)$. We prove the converse inequality. take any geometric cover $\{\Uc_0, \ldots, \Uc_n \}$ with local weak homotopic sections $s_i$. For each $i \in \{0,\ldots,n\}$ we have a homotopy $H \colon \Uc_i \rightarrow \Pa \Ba$ between $P \circ s_i$ and $ \iota_i$, we have that $\pi_1 \circ H= P \circ s_i $ and $\pi_0 \circ H= \iota_i$. Hence, we have the following commutative diagram :
    $$
% https://tikzcd.yichuanshen.de/#N4Igdg9gJgpgziAXAbVABwnAlgFyxMJZABgBoBGAXVJADcBDAGwFcYkQBVAfSxAF9S6TLnyEU5CtTpNW7ADpyACvQAECgEL1+gkBmx4CRCcSkMWbRCAXK1cgKLah+0UQBMpEzTOzLChwKcRQxR3Ki8ZCys5TX4pGCgAc3giUAAzACcIAFskMhAcCCQJaXN2AAlHEAzspABmGgKkABZw0stFSuqcxGLGxCaAqszu93zCxFrBrqRRvuLvSOtVDqnh3Ibx+pAAIxgwKGaATlafEDgeTrXEPL7Rhfk5AHc8RlgVMpVYviA
\begin{tikzcd}
                                                                             & \Pa \E \arrow[r] \arrow[d, "\Pa P"] & \E \arrow[d, "P"] \\
U_i \arrow[r, "H"] \arrow[rru, "s_i", bend left=49] \arrow[ru, "\wtilde H "] & \Pa \Ba \arrow[r]                   & \Ba              
\end{tikzcd}
    $$
     where $\wtilde H  \colon \Uc_i \rightarrow \Pa\E$ is the lift defined by the definiton of weak fibration. Hence we have that $P \circ \pi_1 \circ  \wtilde H = \pi_1 \circ \Pa P \circ \wtilde H= \pi_1 \circ H= \iota_i$ where the third equality follows from Proposition \ref{Conmutativo_caminos}), so $\pi_1 \circ \wtilde H$ is a local weak homotopic section.
\end{proof}

\begin{comment}

\begin{example}
    Let $c \colon \bullet \rightarrow \CC$ be a functor. We have that the existence of a section for $c$ will imply that $\CC$ is also the category $\bullet$. Nonetheless the existence of a weak or Strong homotopic section will be the same as being contractible.

    If we take the category $\mathbb{I}$ defined in Example \ref{Zig-zag_infinity} ad we take $0 \colon \bullet \rightarrow \mathbb{I}$ we have that $\sC(0)=\infty$ since there is no finite cover of $\mathbb{I}$ by trivial categories, i.e. $\mathbb{I}$ is not a finite discrete category. Moreover we have that $\sG_w(0)=0$ since there is a weak homotopic section given that $\mathbb{I}$ is weak-contractible. Moreover we know that $\sG_S(0)=\infty$ since every finite cover of $\mathbb{I}$ will have a subcategory with an infinity number of objects and by a similar argument as the one that we gave in Example \ref{Zig-zag_infinity} this subcategory would not be contractible.  
\end{example}
    
\end{comment}

%The following two results are purely formal and both hold for strong and weak homotopies. In both cases we will denote by $\simeq$ a homotopy equivalence that can be interpreted as the strong or weak one, and we will use $\sG$ for both the strong and weak Svarc Genus. 

\begin{lemma}\label{Composition}
    If $F\colon \CC \rightarrow \Dc$ and $G: \Dc \rightarrow \E$ are two composable functors, we have that $$\sG_w(G) \leq \sG_w (G\circ F)(\mathrm{resp.} \;\sG(G) \leq \sG (G \circ F)).$$
\end{lemma}
\begin{proof}
    Suppose that we have a geometric cover $\{\Uc_0,\dots, \Uc_n\}$ with local weak homotopic sections for $G \circ F$. That means that for every $i\in \{0,...,n\}$ we have this homotopy commutative diagram:
    $$
    % https://tikzcd.yichuanshen.de/#N4Igdg9gJgpgziAXAbVABwnAlgFyxMJZARgBoAGAXVJADcBDAGwFcYkQAdDgYW5AF9S6TLnyEUZYtTpNW7LgBEAxgKEgM2PASJkATNIYs2iThwCiq4ZrFFypfTUNyTXAKpKA+lgHSYUAObwRKAAZgBOEAC2SGQgOBBIuo6yxiAA4pYg4VFIdnEJiLFOqQBimdnRiADMNPGJyUbyHPg49F7lEZU1+bkNziBw7fyU-EA
\begin{tikzcd}
                                              & \CC \arrow[d, "F"] \\
                                              & \Dc \arrow[d, "G"] \\
\Uc_i \arrow[r, "\iota_i"] \arrow[ruu, "s_i"] & \E                
\end{tikzcd}
    $$

This means that $ (G \circ F) \circ s_i \simeq_w \iota_i$ hence $G \circ (F \circ s_i) \simeq_w \iota_i$, therefore $s_i \circ F$ is a local weak homotopic section for $G$.
\end{proof}

\begin{lemma}\label{Homotopy_Invariance_Svarc}
    Let $F \colon \E_1 \rightarrow \E_2$ be a weak homotopy equivalence (resp. strong homotopy equivalence) and $P \colon \E_2 \rightarrow \Ba$ a functor. We have that $\sG(P)_w=\sG_w(P\circ F)$
\end{lemma}
\begin{proof}
    From the previous proposition we know that $\sG(P \circ F) \geq \sG(P)$. For the other inequality assume that you a cover $\{U_0, \dots,U_n\}$ of $\Ba$ with subcategories that have a local weak homotopic section $s_i \colon U_i \rightarrow \E_2 $ for $P$, i.e. $\iota_i \simeq_w P \circ s_i$. Take now a homotopy inverse $G \colon \E_2 \rightarrow \E_1$ of $F$. We claim that $G\circ s_i$ is a local weak homotopic section for $P$. Indeed we have that $P \circ F\circ G  \circ s_i \simeq_w P \circ s_i \simeq_w \iota_i$.
\end{proof}

\section{Motion planning in small categories}\label{sec:motion_planning}

\subsection{The path fibration}
The endpoint functor from the path category to the product of the base with itself plays the role of the classical path fibration in topology. We prove that this functor is indeed a fibration in the sense of Definition \ref{Strong_Homotopy}, providing the key bridge between categorical paths and motion planning.

\begin{teo}\label{Initial_Final_Path_fibration}
    The endpoint functor $\pi \colon \Pa \Ba  \rightarrow \Ba \times \Ba$ defined in Subsection \ref{subsec:motion_planning_functors} is a length-preserving weak fibration (and therefore a strong fibration). 
\end{teo}

\begin{proof}
    Suppose that we have a homotopy $H \colon \CC \rightarrow \Ba \times \Ba$ and $F \colon \CC \rightarrow \mathrm{P} \Ba$. We must define $\hat{H} \colon \CC \rightarrow \mathrm{P} \mathrm{P} \Ba$ such that the following diagram commutes:
    $$
% https://tikzcd.yichuanshen.de/#N4Igdg9gJgpgziAXAbVABwnAlgFyxMJZABgBoBGAXVJADcBDAGwFcYkQAdDgYW5AF9S6TLnyEU5CtTpNW7LgFt6OABYAnBcAAK-AAQAKLgCF6urngXwzHEwEoBQkBmx4CRScWkMWbRJw5KqhraeorK6po61iYOwi5iRABMpJ403nJ+YUGRoTb0sU4iruLIyVRpsr7+JuZYlnDR+fzSMFAA5vBEoABmahAKSGQgOBBIkjI+7AASBb39YzQjSAAsFZOZHGhYAPrEs30DiENLiMkTGf4qysBT-Pvzp4ujiOPpVVkRIdZb94cAzE8VmsLlwfoIegckGcTgDzu8OAAjehqYCgrD8XYgGiMegImCMLRFBJ+NRYNoqHC-QaAxCwvFgKBIABsQze7AAYgJKPwgA
\begin{tikzcd}
                                                                        & \mathrm{P} \mathrm{P} \Ba \arrow[d, "\mathrm{P} \pi"] \arrow[r, "\bar{\pi}_0"'] & \mathrm{P} \Ba \arrow[d, "\pi"] \\
\CC \arrow[r, "H"] \arrow[ru, "\hat{H}"] \arrow[rru, "F", bend left=60] & \mathrm{P} (\Ba \times \Ba) \arrow[r, "\pi_0"]                                  & \Ba\times \Ba                  
\end{tikzcd}
    $$
    This means that we need for each object $c$ in $\CC$ a path of paths $\wH(c)$in $\Ba$. Since we want the diagram to commute $\wH(c)$ has to satisfy some condtions. First, we have that $\bar{\pi}_0\circ \wH=F$, i.e. $\wH(c)(0)=F(c)$ wich means that the first element of the path of paths is the path given by $F$. Furthermore we also want that $\Pa \pi \circ \wH= H(c)$. This means that for each $H(c)=(H_1(c),H_2(c))$ we have that $\wH(c)(i)(0)=H_1(c)(i)$ and that $\wH(c)(i)=H_2(c)(i)$. 

    In order to construct this path of paths we begin with the path $F(c)$:
    $$
    % https://tikzcd.yichuanshen.de/#N4Igdg9gJgpgziAXAbVABwnAlgFyxMJZABgBpiBdUkANwEMAbAVxiRADEAKAYwEpPivALwAJAPoBGHv0EgAvqXSZc+QignkqtRizZc+nAEy95ikBmx4CRQ5ur1mrRCAA6LqBBwIFSy6qIAzHbajnrSnAC2wuKG4bJyWjBQAObwRKAAZgBOEBFIZCA4EEgSPiDZuUi2hcWIpWYVeYjVRUgBCXJAA
\begin{tikzcd}
F(c)(0)=H_1(c)(0) \arrow[r] & F(c)(2) & \dots \arrow[l] \arrow[r] & F(c)(m)=H_2(c)(0)
\end{tikzcd}
    $$
    where the identities come from the conditions that we impose on $F$. Now we need the next path and a morphism between them. The path will be constructed adding the objects $H_j(c)(1)$ at the begining and at the end. The morphism that connects both paths is given by the following diagram:
    $$
    \adjustbox{scale=0.70,center}{
    % https://tikzcd.yichuanshen.de/#N4Igdg9gJgpgziAXAbVABwnAlgFyxMJZAJgBoAGAXVJADcBDAGwFcYkQAJAfQEYAKAMYBKPuSEgAvqXSZc+QigDMFanSat2AMUEie4qTOx4CRACwqaDFm0QgAOnagQcCAyAxH5RAKwW119m5iHVF9aXdZYwUSUh5VKw1bbn5hUMlwjzkTJVj49RsQbVS9dMMs6PM4y3z2BycXUojPbORfKv9EkABJLgAjPgBbXrCyqKIANj8EgqCQsUbMsZRJ9unArmDU+bdFrxQeXOqApN4Qkp3IveQDqiPO5LmRpvKicimak5SRbYzLlrfVh9OKdivpVDAoABzeBEUAAMwAThABkg3iAcBAkDw3IjkUgyOjMYhseFcSjEMpCficUjyeYqYhvDS8YhJgymaTaUgAOw0DFIcbM8mU-mIAAcQp5fKJAE5JeLpUg5ZyWSKidz5Qd2fK0aLTPL6aKeOQdYriST4VzzWbjZqCUaLSAyVi1ViTSryTx7USeIoJJQJEA
\begin{tikzcd}
H_1(c)(0) \arrow[r] \arrow[d] & H_1(c)(0) \arrow[d] & H_1(c)(0) \arrow[r] \arrow[d] \arrow[l] & F(c)(1) \arrow[d] & \dots \arrow[l] & H_2(c)(0) \arrow[l] \arrow[r] \arrow[d] & H_2(c)(0) \arrow[d] \\
H_1(c)(1) \arrow[r]           & H_1(c)(1)           & H_1(c)(0) \arrow[r] \arrow[l]           & F(c)(1)           & \dots \arrow[l] & I_b(mb) \arrow[l] \arrow[r]             & H_2(c)(0)          
\end{tikzcd}}
    $$
    where the first and last morphisms are given by $H(1_c,i_1)$ and we have a repetition at the beginning since all paths begin at $0$. Now the next step is given by:
    $$
    \adjustbox{width=\textwidth}{
    % % https://tikzcd.yichuanshen.de/#N4Igdg9gJgpgziAXAbVABwnAlgFyxMJZARgBoAGAXVJADcBDAGwFcYkQAJAfWIAoBjAJS9igkAF9S6TLnyEUAJgrU6TVu258hvcmMnTseAkQDMymgxZtEIADq2oEHAn0gMhuUQAs51VY1cCgLCuhJSbjJG8iSkxCqW6jaawSJ64e6yxoqx8WrWnDwpoa4ZUaY5Fnns9o7OYQaZ0T5xlf5JgUVpDWUoAKy+CfncQdqi9REeWcj9LX6JBSPCYyWRnijkFXNDhdoKXRONRBtUrfPJo-ula8gAbANV7Yup41dTd7ODAU97EiowUABzeBEUAAMwAThAALZIMggHAQJDkVwQ6GwmgIpAKFGQmGIMzwxGIbHhVF4-qEpBeHFoxAUzGIG40vEAdgxRKZpNxSAJDIAHMykGzKYgAJyCsXspACrm04UMkwS+lE4gSnwi5GyvHEDYi6la2FwhmasHcxA6qXm1UG81KEUy0204i8onim3Oy3EbGUcRAA
\begin{tikzcd}
H_1(c)(1) \arrow[r]           & H_1(c)(1)           & H_1(c)(0) \arrow[l] \arrow[r]           & \dots & H_2(c)(0) \arrow[l] \arrow[r]           & H_2(c)(1)           & H_2(c)(1) \arrow[l]           \\
H_1(c)(2) \arrow[r] \arrow[u] & H_1(c)(1) \arrow[u] & H_1(c)(0) \arrow[l] \arrow[r] \arrow[u] & \dots & H_2(c)(0) \arrow[l] \arrow[r] \arrow[u] & H_2(c)(1) \arrow[u] & H_2(c)(2) \arrow[l] \arrow[u]
\end{tikzcd}}
    $$
    We repeat this until we get to the length $n$ of the path $H(c)$ and we finally obtain:
    $$
    \adjustbox{width=\textwidth}{
    % https://tikzcd.yichuanshen.de/#N4Igdg9gJgpgziAXAbVABwnAlgFyxMJZARgBoAGAXVJADcBDAGwFcYkQAJAfWIAoBjAJS8wAWmKCQAX1LpMufIRQAmCtTpNW7bnyEjRyyTLnY8BIgGY1NBizaIQAHUdQIOBMZAZTiogBZrDTttLmUBYTFDaVkveTMlElJidVstBx1w-QlokwVzFSSUzXtOHkzIoxjvPISrZJti9mdXdxzYn3zkAPqgtNKwvQq26viiAFZA1JLuAYjxStzRlAmeqZDZrIX2mqJyQobg9LLBrZHfFD2qA76MwfnhuPPkADZJxqONsWzPM87X1fe-XKRnUMCgAHN4ERQAAzABOEAAtkgyCAcBAkORPPCkSiaOikMpsQjkYgrGiMYgiTEcaSAhSkGNibjEBMGYhnszSQB2fGUzk0klIckExAADi5SF57IAnJLEDK+UgJYKWdLRRZ5WzRcR5fTRVjVaTiHt2X55cRUQaLaadbqjSjVOyVbChYhiCLKXKHe7PY6pJQpEA
\begin{tikzcd}
H_1(c)(n-1) \arrow[r]         & H_1(c)(n-1)                     & H_1(c)(n-2) \arrow[l] \arrow[r] & \dots & H_2(c)(n-2) \arrow[l] \arrow[r]           & H_2(c)(n-1)           & H_2(c)(n-1) \arrow[l]         \\
H_1(c)(n) \arrow[r] \arrow[u] & H_1(c)(n-1) \arrow[r] \arrow[u] & H_1(c)(n-2) \arrow[r] \arrow[u] & \dots & H_2(c)(n-2) \arrow[l] \arrow[r] \arrow[u] & H_2(c)(n-1) \arrow[u] & H_2(c)(n) \arrow[l] \arrow[u]
\end{tikzcd}}
    $$
    This construction can be extended to a functor where we take the path morphism $$H(f)=(H(f_1),H(f_2) \colon (H_1(c),H_2(c)) \rightarrow (H_1(d),H_2(d))$$ to a morphism between path of paths in the following way. $H(f)$ is a zig zag of arrows by the construction of the localization.
    
    Suppose that we have a direct morphism in the sense that $H(f)=((F_1,\alpha_1),(F_2,\alpha_2))$ where $(F_j,\alpha_j) \colon H_j(c)\rightarrow H_j(d)$ and $j \in \{1,2\}$. We can construct a direct morphism in $\Pa \Pa \Ba$. For each $i$ we have that there $\alpha_j(i) \colon H_j(c)(i) \rightarrow H_j(d)(i)$ so we can define a morphism between $\wH(c)(i) \rightarrow \wH(d)(i)$ as follows:
    $$
    \adjustbox{width=\textwidth}{% https://tikzcd.yichuanshen.de/#N4Igdg9gJgpgziAXAbVABwnAlgFyxMJZARgBoAGAXVJADcBDAGwFcYkQAJAfWIAoBjAJS8sAWmKCQAX1LpMufIRQAmCtTpNW7bnyEjRyyTLnY8BIgGY1NBizaIQAHUdQIOBMZAZTiogBZrDTttLmUBYTFDaVkveTMlZABWQNstB24wvTEJaJMFcxRyFM17Th5wkSMY73yEgDZi4PTQiqwqvPiiIuJ1VNKdXigI9tifApJSHpsSkL4h-RzPGs6UZKmgtLKw+eyR5d8UBvW+kO3h3NHa-0nemeaz-SiluIPkVWO7rcGIgz2X8asHyaThcbg86hgUAA5vAiKAAGYAJwgAFskGQQDgIEhyJ4kaj0TQsUhlHjkWjEFZMdjEKSYviKVTiYg-GSCYhktScWyKQ0uSyeUhOcyAOzTYHOJhoAAW9HKbQuDJxRJpAA5xZtJYwZXK+LtFeTCfziKoNqUtTr5b8DeyxfzVYLECaVUgHfTDU7TcziBZHUyacQMScHBbZS1IiMlU6Md7fe72YGXYgAJyO4hFfmp+MUgKZjXmxxSsNhfWOvnejPBkFFuUloyUKRAA
\begin{tikzcd}
H_1(c)(i) \arrow[r] \arrow[d, "\alpha_1(i)"] & H_1(c)(i-1) \arrow[d, "\alpha_1(i-1)"] & H_1(c)(i-2) \arrow[l] \arrow[r] \arrow[d, "\alpha_1(i-2)"] & \dots & H_2(c)(i-2) \arrow[l] \arrow[r] \arrow[d, "\alpha_2(i-2)"] & H_2(c)(i-1) \arrow[d, "\alpha_2(i-1)"] & H_2(c)(i) \arrow[l] \arrow[d, "\alpha_2(i)"] \\
H_1(d)(i) \arrow[r]                          & H_1(d)(i-1)                            & H_2(d)(i-2) \arrow[l] \arrow[r]                            & \dots & H_2(d)(i-2) \arrow[l] \arrow[r]                            & H_2(d)(i-1)                            & H_2(d)(i) \arrow[l]                         
\end{tikzcd}}
    $$
    We can glue all this morphis in a morphism between path of paths. Moreover, this clearly defines a functor between the localizations since this construction preserves weak equivalences.
\end{proof}

The previous result is the reason why  the endpoint functor is also referred to as the {\em path fibration.}

\begin{corollary}
    The functors $\pi_0$ and $\pi_1$ are weak fibrations.
\end{corollary}

\begin{proof}
    It follows from Corollary \ref{Projection} and Proposition \ref{Composition_Fibration} since $\pi_i= p_i\circ \pi$.
\end{proof}

\subsection{LS-Category, Categorical Complexity and homotopic distance}

We introduce the novel notions of weak Lusternik-Schnirelmann category, weak categorical complexity and weak homotopic distance of small categories, while reviewing their strong counterparts introduced by Tanaka \cite{TANAKA} and the two last authors of this paper \cite{MAC-MOSQ-FUNCT}, respectively. 

\begin{defi}
    Let $\CC$ be a small category. The \emph{weak (resp. Strong) Lusternik-Schnirelmann category} of $\CC$, denoted by $\ccat_w(\CC)$ (resp. $\ccat(\CC))$, is the least non-negative integer $n \in \mathbb{Z}$ such that there is a geometric cover $\{\Uc_0,...,\Uc_n\}$ of $\CC$ verifying that for every $i\in \{0,...,n\}$ the inclusion $\iota_i\co \Uc_i \rightarrow \CC$ is weak (resp. strong) homotopic to a constant functor. 
\end{defi}

Let $\CC$ be a small category. We define the diagonal functor $\Delta     \colon     \CC \rightarrow \CC \times \CC$ as the functor that takes each object $c$ to $\Delta(c)=(c,c)$ and each morphism $f$ to $\Delta(f)=(f,f)$.

\begin{definition}
Let $\CC$ be a a small category. A subcategory $\Uc$ of  $\CC \times \CC$ is a {\em weak (resp. strong) Farber subcategory} if there is a functor $F    \colon     \Uc  \rightarrow \CC$ such that $\Delta \circ F$ is weak (resp. strong) homotopic to $\iota$ where $\iota\colon \Uc \rightarrow \CC \times \CC$ is the inclusion.
\end{definition}

\begin{defi}\label{Categorical_Complexity}
Let $\CC$ be a small category, the {\em weak (resp. strong) categorical complexity} of $\CC$, denoted by $\cTC(\CC)$ (resp. $\cTC_w(\CC)$), is the least non-negative integer $n \in \mathbb{Z}$ such that there is a geometric cover $\{\Uc_0,\dots,\Uc_n\}$ of $\CC \times \CC$  by Farber subcategories. 
\end{defi}

\begin{defi}\label{Homotopic_distance} Let $\CC$ and $\Dc$ be two small categories and let $F,G    \colon \CC    \rightarrow \Dc$ be two functors. The {\em weak (resp. strong) homotopic distance } $\cD(F,G)$ ($\cD_w(F,G)$) between $F$ and $G$ is the least non-negative integer $n \in \mathbb{Z}$ such that there is a geometric cover $\{\Uc_{0},\dots,\Uc_n\}$ such that $F|_{\Uc_i} \simeq_w G|_{\Uc_i}$ ( $F|_{\Uc_i} \simeq G|_{\Uc_i}$)for every $ 0 \leq i \leq  n$. 
\end{defi}

The last two authors of the present paper proved (see \cite{MAC-MOSQ-FUNCT}):
\begin{enumerate}
    \item Let $\CC$ be a connected category and $c \colon \CC \rightarrow \CC$ be a constant functor. Then
    $\ccat(\CC)= \cD(c,1_{\CC})$
    where $1_{\CC}$ is the identity.
    \item The strong categorical complexity of a small category $\CC$ is the strong homotopic distance between the two projections $p_1, p_2 \colon \CC \times \CC \rightarrow \CC$, that is, $\cTC(C) = \cD(p_1,p_2).$
\end{enumerate}

The analogous result for the weak Lusternik-Schnirelmann category follows immediately from the definition:

\begin{prop}\label{pro:LS-Category_Homotopic-Distance_weak}
     Let $\CC$ be a connected category and $c \colon \CC \rightarrow \CC$ be the constant functor with value the object $c$ in $\CC$. We have that:
    $$\mathrm{ccat}_w(\CC)= \cD_w(c,1_{\CC})$$
    where $1_{\CC}$ is the identity.
\end{prop}

We prove the analogous result for the weak categorical complexity.

\begin{prop}\label{pro:Complexity_Homotopic_Distance_weak}
    The weak categorical complexity of $\CC$ is the weak homotopic distance between the two projections $p_1, p_2 \colon \CC \times \CC \rightarrow \CC$, that is, $$\mathrm{cTC}_w(C) = \cD_w(p_1,p_2).$$
\end{prop}

\begin{proof}
    Take a Farber subcategory $\Uc$ with inclusion $\iota \colon \Uc \rightarrow \CC \times \CC$. Since $\Uc$ is Farber there is a functor $F \colon \Uc \rightarrow \CC$ and a weak homotopy $H \colon \Uc \rightarrow \Pa (\CC \times \CC)$ between $\iota$ and $\Delta \circ F$, i.e. $\pi_0 \circ H=\iota$ and $\pi_1 \circ H= \Delta \circ F$. We have to define a weak homotopy $H' \colon \Uc \rightarrow \Pa\CC$ between $p_1$ and $p_2$, which is the same as giving a functorial form of assigning a path between $c$ and $d$ for any object $(c,d)$ in $\Uc$.
    
    For each object $(c,d)$ in $\Uc$ we have that $H(c)$ is a path between $(c,d)$ and $\Delta \circ F(c,d)=(F(c,d),F(c,d))$ and this is the same as a path from $I_c \colon \mathbb{I}_m \rightarrow \CC \times \CC$ from $c$ to $F(c,d)$ and a path $J_D \colon \mathbb{I}_m \rightarrow \CC \times \CC$ from $d$ to $F(c,d)$. We can take the composition of $I$ with the inverse of $J$ to obtain the wanted path. Moreover, this is functorial since for each morphism $(f,g) \colon (c,d) \rightarrow (a,b)$ we can define a morphism between them using the fact that there is a path morphism $H(f,g) \colon H(c,d) \rightarrow H(a,b)$ with means that there a morphism between $I_c$ and $I_a$ and between $J_d$ and $J_b$ that induces our wanted morphism between $H'(c,d)$ and $H'(a,b)$.

    Now take a subcategory $\Uc$ with a weak homotopy $H \colon \Uc \rightarrow \Pa\CC$ between $p_1$ and $p_2$. We can show that $\Uc$ is a Farber subcategory with the functor $F \colon \Uc \rightarrow \CC$ given by $p_1$, i.e. $\iota \simeq_w \Delta \circ p_1$. A homotopy $H' \colon \Uc \rightarrow$ between $\iota$ and $\Delta \circ p_1$ is a functorial assignment that takes any object $(c,d)$ to a path between $(c,d)$ and $\Delta \circ p_1(c,d)=(c,c)$. This can be done by using the fact that a path in $\CC \times \CC$ is the same as two paths in $\CC$. Hence we can take the path between $c$ and $d$ defined by $H(c,d)$ and the constant path between $c$ and $c$ with the same length as $H(c,d)$. This is functorial since $H$ is a functor.
\end{proof}

\begin{remark}\label{rmk:weak_Strong_Homotopy_Invariants}
    Observe that as a consequence of Proposition \ref{prop:Strong_Homotopic_entails_weak_Homotpic} and Proposition \ref{weak_Strong_Finite}, for each pair of functors $F,G \colon \CC \rightarrow \Dc$,  $\cD_w(F,G) \leq \cD (F,G).$ Therefore, $\ccat_w(\CC) \leq \ccat(\CC)$ and $\cTC_w(\CC) \leq \cTC(\CC)$. Furthermore, all inequalities are equalities if $\CC$ is finite by Proposition \ref{weak_Strong_Finite}.
\end{remark}

\subsection{Relations between Svarc Genus, sectional category and other homotopy invariants}
We prove a result, which may be seen as a counterpart of \cite[Theorem 2.8]{MAC-MOSQ} in the context of small categories, relating the previous homotopic invariants.

\begin{prop}\label{prop:secat_decreases_along_pullback}
    For every pullback of small categories:
    $$
    % https://tikzcd.yichuanshen.de/#N4Igdg9gJgpgziAXAbVABwnAlgFyxMJZABgBoBGAXVJADcBDAGwFcYkQAdDgIXoHIQAX1LpMufIRRli1Ok1bsuAUQHDR2PASLkKshizaJOPekJEgMGidtIya+hUeVDZMKAHN4RUADMAThAAtkg6IDgQSGRyBuwACqrm-kGRNOFIAEz28oYgAOJmvgHBiADMqRGImdGOILEFIEnFoWmlWTFGuaqUgkA
\begin{tikzcd}
\E' \arrow[d, "P'"] \arrow[r, "G'"] & \E \arrow[d, "P"] \\
\Ba' \arrow[r, "G"]                 & \Ba              
\end{tikzcd}
    $$
    we have that:
    $$\sC (P') \leq \sC(P).$$
    In particular, if $P$ is a weak (resp. strong) fibration, then
    $\sG_w(P') \leq \sG_w(P)$ (resp.$\sG_s(P) \leq \sG_s(P')$).
\end{prop}

\begin{proof}
    Suppose that we have a geometric cover $\{\Uc_0, \dots \Uc_n\}$ of $\Ba$ with local sections $s_i \colon \Uc_i \rightarrow \E$. For each $0 \leq i \leq n$, we define $G^{-1}(\Uc_i)$ as the subcategory of $\Ba'$ with objects $b$ and morphisms $f$ in $\Ba'$ such that $G(b)$ and $G(f)$ are objects and morphisms in $\Uc_i$. %This defines a subcategory since every identity morphism of an object is also a morphism and the composition of two morphism that lies in $\Uc_i$ lies also in $\Uc_i$ by functoriality and the fact that $\Uc_i$ is a subcategory. 
    Then $\{G^{-1}(\Uc_0), \dots , G^{-1}(\Uc_n)\}$ is a geometric cover of $\Ba'$. %Indeed, take any chain of composable morphism $(f_1,\dots,f_n)$ in $\Ba'$ we have that $(G(f_1),\dots,G(f_n))$ is a chain of composable morphism in $\Ba$ and therefore there is $0 \leq i\leq n$ such that the chain is inside $\Uc_i$, hence $(f_1,\dots,f_n)$ lies inside $G^{-1}(\Uc_i)$.

    We define local sections $s'_i \colon G^{-1}(\Uc_i) \rightarrow \E'$ as $s_i'(b)=(b,s_i(G(b))$. This is well defined since $P(s (G(b)))=G(b)$ and it is clearly a local section using the fact that $P'(b,s_i(G(b))=b$.

    If $P$ is a weak (strong) fibration then $P'$ is also a weak (strong) fibration (Proposition \ref{Pullback}) and $\sC(P)= \sG_w(P) (\mathrm{resp.} \,\sC(P)= \sG(P))$ and $\sC(P')= \sG_w(P') (\mathrm{resp.} \, \sC(P')= \sG_(P'))$ (Proposition \ref{prop:homotopic-sections_fibration_are_sections}). %using the fact that the pullback of a strong (weak) fibration is a strong (weak) fibration (Proposition \ref{Pullback}) and the fact that in a strong (weak) fibration we can lift every local homotopy section to a local section that appears in Proposition \ref{prop:homotopic-sections_fibration_are_sections}.
\end{proof}

\begin{teo}\label{teo:Svarc_Genus_Homotopic_Distance}
Let $F,G \colon \CC \rightarrow \Dc$ be two funtors and let $\pi \colon \Pa \Dc \rightarrow \Dc \times \Dc$ be the path fibration. Consider the pullback:
$$
% https://tikzcd.yichuanshen.de/#N4Igdg9gJgpgziAXAbVABwnAlgFyxMJZABgBpiBdUkANwEMAbAVxiRAAUQBfU9TXfIRQBGclVqMWbADrT2dAASyAIgGNuvEBmx4CRMsPH1mrRCFkBhCxr47BRUYerGpZlaqXS8AW3ie13OIwUADm8ESgAGYAThDeSGQgOBBIwjxRsfGIiclIAEzOkqYgAI42IDFxqdS5iADMhSYy0mhY5ZVZBUkp9Y2uIAAUAGKkAOIAlIFcQA
\begin{tikzcd}
P \arrow[r, "p"] \arrow[d, "q"] & \Pa \Dc \arrow[d, "\pi"] \\
\CC \arrow[r, "{(F,G)}"]   & \Dc \times \Dc          
\end{tikzcd}
$$
Then, $\cD_w(F,G)=\sC(q)$ and $\sC(q)\leq \cD(F,G)$.  Moreover, if $\CC$ is finite, $\sC(q)=\cD(F,G).$
\end{teo}

\begin{proof}
   We begin by proving that $\sC(q)=\cD_w(F,G)$. First,   $\sC(q)=\sG_w(q)=\sG(q)$ by Proposition \ref{Pullback}, Proposition \ref{prop:homotopic-sections_fibration_are_sections} and Theorem \ref{Initial_Final_Path_fibration}.  Hence, it is enough to prove the inequalities $\sC_w(q) \leq \cD_w (F,G)$ and  $\cD_w(F,G) \leq \sG(q) $.

    We begin by proving the first inequality. Suppose that we have a geometric cover $\{U_0, \dots, U_n\}$ such that $F|_{\Uc_i} \simeq_w G|_{\Uc_i}$ for every $ 0 \leq i \leq  n$. We define the local sections $s_i \colon \Uc_i \rightarrow P$ as:
    $$s_i(c)=(c,H(c))$$ where $H \colon \Uc_i  \rightarrow \Pa \Dc$ is the weak homotopy between $F|_{\Uc_i} $ and $G|_{\Uc_i}$. 
    
    Conversely, suppose that we have a geometric cover $\{\Uc_0, \dots, \Uc_n\}$ with local sections $s_i \colon U_i \rightarrow P$ for every $ 0 \leq i \leq  n$. Take $i$ such that $ 0 \leq i \leq  n$ we must define a homotopy between $F|_{\Uc_i} $ and $G|_{\Uc_i}$ but this can easily be done usign $s_i$. Indeed, we define $H$ as $p \circ s_i$. We claim that this give us the wanted homotopy since we have that $\pi \circ H=\pi \circ p \circ s_i=(F,G) \circ q \circ s_i = (F,G) \circ \iota_i $. If we compose with the natural projections $p_1, p_2 \colon \Dc \times \Dc \rightarrow \Dc$ we have that $\pi_j \circ H= p_j \circ \pi \circ H= p_j \circ (F,G) \circ \iota_i$ that equals $F|_{\Uc_i}$ if $j=1$ and $ G|_{\Uc_i}$ if $j=2$.

    The last inequality and the equality under finiteness hypothesis follows from Remark \ref{rmk:weak_Strong_Homotopy_Invariants}.
\end{proof}

\begin{corollary}\label{cor:cTC-Sectional-Category}
 Let $\CC$ be a small category and $\pi \colon \Pa \CC \rightarrow \CC \times \CC$ be the path fibration. Then 
    $\mathrm{cTC}_w(\CC) = \sC(\pi)$ and $\sC(\pi) \leq \mathrm{cTC}(\CC)$ with a equality if $\CC$ is finite.
\end{corollary}
\begin{proof}
The morphism $(\pi_1,\pi_2) \colon \CC \times \CC \rightarrow \CC \times \CC$ is the identity of $\CC \times \CC$. So we have that
$$% https://tikzcd.yichuanshen.de/#N4Igdg9gJgpgziAXAbVABwnAlgFyxMJZABgBoBGAXVJADcBDAGwFcYkQAdDgYW67wC28AARdeIAL6l0mXPkIpyFanSat2Y7qI6CRmydJAZseAkSXEVDFm0ScOABXrbxUmSflEylmtfV2uJxduSRUYKABzeCJQADMAJwgBJDIQHAgkJVUbdgAKLjQsAH0lAuKAJgBKAzjE5MRymnTM3zVbe0KakASkpABmJozEYjduuv7BpHKJSgkgA
\begin{tikzcd}
\Pa \CC \arrow[d] \arrow[r]                & \Pa \CC \arrow[d, "\pi"] \\
\CC\times \CC \arrow[r, "{1_{\CC \times \CC}}"] & \CC \times \CC          
\end{tikzcd}
$$
and therefore applying Proposition \ref{teo:Svarc_Genus_Homotopic_Distance} and Proposition \ref{pro:Complexity_Homotopic_Distance_weak} we obtain $\mathrm{cTC}_w(\CC)=\cD_w(\pi_1,\pi_2)=\sC(\pi)$. The last part is also a consequence of Remark \ref{rmk:weak_Strong_Homotopy_Invariants}.
\end{proof}

\begin{corollary}
Let $\CC$ be a small connected category and let $c \colon \bullet \rightarrow \CC$ be a constant functor. Then 
    $\mathrm{ccat}_w(\CC) = \sG_w(c)$ 
    and $\mathrm{ccat}_w(\CC) \leq \mathrm{cTC}_w(\CC)$.    
\end{corollary}
\begin{proof}
     Applying Theorem \ref{teo:Svarc_Genus_Homotopic_Distance} and Proposition \ref{pro:LS-Category_Homotopic-Distance_weak} we obtain that $\mathrm{ccat}_w(\CC)=\cD_w(c,1_\CC)=\sC(q)$ where $q$ is the morphism that appears in the pullback
     $$
     % https://tikzcd.yichuanshen.de/#N4Igdg9gJgpgziAXAbVABwnAlgFyxMJZABgBoBGAXVJADcBDAGwFcYkQAdDgYW5AF9S6TLnyEU5CtTpNW7LrwAEXPAFt4ynn0HDseAkUnFpDFm0ScOABXqbeAoSAx6xRMsZqm5FrjbvbpGCgAc3giUAAzACcIVSQyEBwIJEkZM3YACnIAfQVuUgBjAEoHSJi4xAAmGiSUz1lzSzQsUpBo2KQAZhrkxGIdNvKunqRK-kp+IA
\begin{tikzcd}
P \arrow[d,"q"] \arrow[r]  & \Pa \CC \arrow[d, "\pi"] \\
\CC \arrow[r, "{(c,1_\CC)}"] & \CC \times \CC          
\end{tikzcd}$$ 
We claim that $P$ is precisely (with the notation of Theorem \ref{functor_factorisation_through_fibration}) $\mathrm{E}_c$ where $c$, by an abuse of notation, denotes the constant functor associated to the object $c$. Indeed, the objects of $P$ are pairs $(x,I)$ where $x$ is an object in $\CC$ and $I$ is a path between $c$ and $x$. Using the notation of Theorem \ref{functor_factorisation_through_fibration} of $c_1$ as the homotopy equivalence and $c_2$ as the weak fibration of the fibration replacement of $c$ we can check that $c_2=q$. Hence, $\sG_w(c)=\sG_w(c_1 \circ c_2)=\sG_w(c_2)=\sG_w(q)=\sC(q)=\cD_w(c,1_\CC)$.
     %This follows from Propositions \ref{prop:homotopic-sections_fibration_are_sections}, \ref{LS-CAT} and \ref{Composition}.
     The inequality $\mathrm{ccat}_w(\CC) \leq \mathrm{cTC}_w(\CC)$ follows from Proposition \ref{prop:secat_decreases_along_pullback}.
\end{proof}

\begin{prop}\label{prop:secat_less_cat}
    Let $P \colon \E \rightarrow \Ba$ be a surjective weak fibration. Then $\sC(P) \leq \ccat_w(\Ba) $
\end{prop}

\begin{proof}
    Suppose that you have a geometric cover $\{U_0,\dots,U_n\}$ of $\Ba$ such that $\iota_i\simeq_w c $ where $c$ is a constant functor. We can define $s_i \colon \Uc_i \rightarrow \E$ as a constant functor $e$ with $e$ such that $P(e)=b$. Then we have that $s_i$ is a local section since $P \circ s_i=c \simeq_w \iota_i$.
    %\colon \Uc_i \rightarrow \Ba$
\end{proof}

\begin{corollary}
     Let $F,G \colon \CC \rightarrow \Dc$ be a functor between small categories. Then:
     \begin{enumerate}
         \item $\cTC(\CC)_w \leq \ccat_w(\CC \times \CC)$
         \item $\cD_w(F,G) \leq \ccat_w(\CC)$
     \end{enumerate}
\end{corollary}

\begin{proof}
 We prove (1). We know that $\cTC(\CC)_w=\sC(\pi)$ by Proposition \ref{cor:cTC-Sectional-Category}. Applying Proposition \ref{prop:secat_less_cat} we obtain $\sC(\pi) \leq \ccat_w(\CC \times \CC)$.
 We prove (2). $\cD_w(F,G)=\sC(q)$ for $q$ the morphism defined by the pullback of Theorem \ref{teo:Svarc_Genus_Homotopic_Distance}. Applying Proposition \ref{cor:cTC-Sectional-Category} we obtain the inequality since $\sC(q) \leq \ccat_w(\CC)$.
\end{proof}

\begin{comment}
 
Now we will get some lower bounds for the LS-Category and the categorical complexety using a variation of following theorem that we proved in \cite{Cohomologysvarcgenus}.

\begin{teo}\cite[Theorem 2.18]{Cohomologysvarcgenus}
    Let $P \colon \E \rightarrow \Ba$ be a weak fibration and let $D$ be a natural system in $\Ba$ with an endopairing. The length of the cup product in $\ker{P^*}$, denoted by $\cpl(\ker{P^*})$, is a lower bound for the weak Svarc genus of $P$:  $\cpl (\ker{P^*}) \leq \sG_w(P)$.
\end{teo}
\begin{proof}
    The proof in \cite{Cohomologysvarcgenus} show that the number of local section is bounded below by the length of the cup product. Since i a weak fibration the number of local sections is the same as the number of weak-local section we have the wanted inequality.
\end{proof}

\begin{corollary}
    Let $\CC$ be a small category and $D\colon \CC \rightarrow \Ab$ a natural system. We have the following inequalities:
    $$\cpl(H^*(\CC,D) \leq \mathrm{ccat}_w(\CC) \leq \mathrm{ccat}(\CC)$$
\end{corollary}

\begin{proof}
    We know that $\ccat_w(\CC)=$
\end{proof}

%The previous example not only as a way of sho

\end{comment}

%\bibstyle{sn-mathphys}
\bibliographystyle{plain}
\bibliography{biblio}% common bib file HACER ESTE FICHERO

%\input sn-article.bbl


\end{document}